\documentclass[reqno,oneside]{amsart}
\usepackage{amssymb,bm}
\usepackage{amsmath,mathtools}

\usepackage[colorlinks=true]{hyperref}
\hypersetup{urlcolor=blue, citecolor=red}

\newtheorem{Theorem}{Theorem}[section]
\newtheorem{Corollary}[Theorem]{Corollary}
\newtheorem{Lemma}[Theorem]{Lemma}
\newtheorem{Proposition}[Theorem]{Proposition}

\newtheorem{Remark}[Theorem]{Remark}

\numberwithin{equation}{section}

\makeatletter
\def\mathcenterto#1#2{\mathclap{\phantom{#1}\mathclap{#2}}\phantom{#1}}
\let\old@widetilde\widetilde
\def\widetildeto#1#2{\mathcenterto{#2}{\old@widetilde{\mathcenterto{#1}{#2}}}}
\let\old@widehat\widehat
\def\widehatto#1#2{\mathcenterto{#2}{\old@widehat{\mathcenterto{#1}{#2\,}}}}
\makeatother

\def\widetilde{\widetildeto{K}}

\def\C {\mathbb C}
\def\R {\mathbb R}

\newcommand{\Op}{\operatorname{Op}}
\newcommand{\supp}{\operatorname{supp}}
\newcommand{\<}{\langle}
\renewcommand{\>}{\rangle}
\newcommand{\WF}{\operatorname{WF}}
\newcommand{\id}{\operatorname{Id}}
\newcommand{\bid}{\operatorname{\bf Id}}

\newcommand{\p}{\partial}
\newcommand{\Vol}{\operatorname{Vol}}
\newcommand{\diag}{\operatorname{diag}}
\newcommand{\tr}{\operatorname{tr}}
\renewcommand{\Re}{\operatorname{Re}}
\renewcommand{\Im}{\operatorname{Im}}

\newcommand{\bI}{\mathbf I}
\newcommand{\bN}{\mathbf N}

\newcommand{\cA}{\mathcal A}

\newcommand{\cE}{\mathcal E}
\newcommand{\cG}{\mathcal G}

\newcommand{\cI}{\mathcal I}

\newcommand{\cP}{\mathcal P}
\newcommand{\cS}{\mathcal S}

\newcommand{\tbN}{\widetilde \bN}

\newcommand{\tbI}{\tilde \bI}

\newcommand{\tU}{\widetilde U}
\newcommand{\tM}{\widetilde M}

\begin{document}
\title[The attenuated geodesic ray transform on tensors]{The attenuated geodesic ray transform on tensors: generic injectivity and stability}

\author[Yernat M. Assylbekov]{Yernat M. Assylbekov }
\address{Department of Computational Mathematics, Science and Engineering, Michigan State University, East Lansing, MI 48824, USA}
\email{y\_assylbekov@yahoo.com}

\maketitle

\begin{abstract}
We consider the attenuated geodesic ray transform defined on pairs of symmetric $2$-tensors and $1$-forms on a simple Riemannian manifold. We prove injectivity and stability results for a class of generic simple metrics and attenuations containing real analytic ones. In fact, methods used in this paper can be modified to generalize our results for a class of non-simple manifolds similar to Stefanov-Uhlmann [American Journal of Mathematics, {\bf 130} (1):239--268 (2008)].
\end{abstract}

%----------------------------------------------------------------------------------------------------------------------------------------------------------------

\section{Introduction and main results}\label{sctn::introduction}

Consider a smooth compact $n$-dimensional Riemannian manifold $(M,g)$ with smooth boundary $\p M$. Let $SM$ be its unit sphere bundle and $\p_\pm SM$ be the set of inward/outward unit vectors on~$\p M$,
$$
\p_\pm SM:=\{(x,v)\in SM:x\in\p M\text{ and }\pm\<v,\nu(x)\>_{g(x)}\ge 0\},
$$
where $\nu$ is the inward unit normal to $\p M$. For a given $(x,v)\in SM$, $\gamma_{x,v}$ is the unique geodesic with $x=\gamma_{x,v}(0)$, $v=\dot\gamma_{x,v}(0)$ and $\tau(x,v)$ is the first positive time when it exits $M$. Throughout the paper, we assume that $(M,g)$ is \emph{\bfseries simple}, meaning that $\p M$ is strictly convex and that any two points on $\p M$ are joined by a unique minimizing geodesic. The notion of simplicity naturally arose in the context of the boundary rigidity problem \cite{michel1981rigidite}. In particular, simplicity implies that $M$ is simply connected and $\tau$ is a bounded function on $SM$.

The \emph{\bfseries attenuated geodesic ray transform} of $f\in C^\infty (SM;\C)$, with attenuation $a\in C^\infty(M;\C)$, is given by
$$
I_{a} f(x,v):=\int_0^{\tau(x,v)} \exp \bigg(\int_0^t a(\gamma_{x,v}(s))\, ds\bigg)f(\gamma_{x,v}(t),\dot\gamma_{x,v}(t))\,dt,\quad (x,v)\in\p_+SM.
$$
It is clear that a general function $f\in C^\infty(SM;\C)$ cannot be determined by its attenuated geodesic ray transform, since $f$ depends on more variables than $I_a f$. Moreover, one can easily see that the functions of the type $Xu$ with $u|_{\p(SM)}=0$ are always in the kernel of $I_a$. However, in applications one often needs to invert the transform $I_a$ acting on functions on $SM$ arising from symmetric tensor fields. Further, we will consider this particular case.

We denote by $S^2_M$ and $\Lambda^1_M$ the bundles of complex-valued symmetric $2$-tensors and $1$-forms on $M$, respectively. For the restrictions of $I_a$ to $C^\infty(M;\C)$, $C^\infty(M;\Lambda^1_M)$ and $C^\infty(M;S^2_M)$ we will use the notations $I_a^0$, $I_a^1$ and $I_a^2$, respectively. By $\bI_{a}$ we denote the following operator
$$
\bI_{a} [f,\alpha]:=I_a^2 f+I_a^1 \alpha,\qquad [f,\alpha]\in C^\infty(M;S^2_M\times \Lambda^1_M).
$$
The domain of $\bI_a$ can be extended to $L^2(M;S^2_M{\times}\Lambda^1_M)$; see Section~\ref{sec:adjoints} for details. We also define $\cI_a$ as $\cI_{a} [w,\phi]:=I_a^1 w+I_a^0 \phi$ for $[w,\phi]\in C^\infty(M;\Lambda^1_M\times\C)$. Then $\cI_a$ is particular case of $\bI_a$ since, for $[\phi,\alpha]\in C^\infty(M;\C\times\Lambda^1_M)$, one can write $\cI_a[w,\phi]:=\bI_a[\phi g,w]$.

Various cases of $\cI_a$ have applications in imaging techniques such as SPECT \cite{budinger1979emission} and Doppler tomography \cite{juhlin1992principles}. It also appeared in the context of anisotropic inverse conductivity problem of Calder\'on \cite{ferreira2009limiting} on so-called \emph{\bfseries admissible manifolds}, i.e. compact Riemannian manifolds with boundary which are conformally embedded in a product of the Euclidean line and a simple manifold. In \cite{ferreira2009limiting}, unique determination of the conductivity from the boundary measurements was reduced to injectivity of $\cI_a$. In a similar way the latter is related to inverse problems for other elliptic equations and systems \cite{assylbekov2017polyharmonicadmissible,kenig2011inverse,krupchyk1702inverse,krupchyk2017inverse} including nonlinear ones \cite{assylbekov2017kerrinverse}. The transform $\bI_a$ arises in several problems as well. Namely, boundary and lens rigidity problems \cite{sharafutdinov1994integral,stefanov2008microlocal,stefanov2008boundary} and inverse boundary value problems for the Hodge Laplacian \cite{chung2017hodge}. In forthcoming works we demonstrate two applications of $\bI_a$. In the first one, it will be illustrated that unique determination of coefficients of polyharmonic operators with second order perturbation from Dirichlet-to-Neumann on admissible manifolds can be reduced to injectivity of $\bI_a$, generalizing results of \cite{ghosh2017inverse}. In the second one, we will show application of $\bI_a$ in the linearized anisotropic Calder\'on's problem posed in \cite{sharafutdinov2007variations}. We believe that $\bI_a$ will find applications in other inverse boundary value problems as well.

The problem of injectivity of $\bI_a$ has a natural obstruction. Indeed, the kernel of $\bI_a$ has a non-trivial elements, since, as one can easily see, $\bI_ad_a[w,\phi]=0$ for all $[w,\phi]\in C^\infty(M;\Lambda^1_M\times\C)$ with $[w,\phi]|_{\p M}=0$, where
$$
d_a[w,\phi]:=[d^s w+a\phi g,d\phi+aw],\quad (d^s w)_{ij}:=(\nabla_i w_j+\nabla_j w_i)/2.
$$
We say that $\bI_a$ is \emph{$s$-injective} if these are the only elements of the kernel. Then the inverse problem we consider is whether $\bI_a$ is $s$-injective.

In the case $a=0$, the problem is known as the \emph{\bfseries tensor tomography problem} which received considerable interest \cite{paternain2013tensor,paternain2015invariant,sharafutdinov1994integral,stefanov2004stability,stefanov2005boundary,stefanov2008nonsimple,stefanov2014inverting,uhlmann2016inverse}. The latter problem consists of determining a tensor field from its geodesic ray transform (with no attenuation). The reader is referred to the survey articles \cite{paternain2014tensor,stefanov2008microlocal,stefanov2008boundary} for the most recent developments in this direction. For $a\not\equiv 0$, this problem was studied in two dimensions. On simple surfaces, $s$-injectivity was proven in \cite{ainsworth2013attenuated} (see Remark~7.5 therein) following \cite{paternain2012attenuated,paternain2013tensor}. Inversion formulas/procedure were given on Euclidean unit disc \cite{monard2017efficient} and on simple surfaces \cite{monard2016inversion}. Range characterization of $\bI_a$ was studied in Euclidean case \cite{sadiq2016tensortransform} and on simple surfaces \cite{ainsworth2015range}.

In the present paper, we are interested in proving injectivity results and stability estimates for the transform $\bI_a$. Focusing in the real-analytic setting, we use analytic microlocal analysis which was developed in \cite{stefanov2004stability,stefanov2005boundary,stefanov2008nonsimple} for the tensor tomography. This method, which goes back to Guillemin and Sternberg \cite{guillemin1979some}, led to many injectivity results of various types of ray transforms in the real-analytic category \cite{abhishek2017support,dairbekov2007boundary,frigyik2008x,holman2013generic,holman2010doppler,zhou2017generic}.

We now state the main results and give an outline of the remainder of the article. Our first main result is the following injectivity result for $\bI_a$.

\begin{Theorem}\label{thm::main 1}
Let $(M,g)$ be a real analytic simple manifold. Suppose that $a:M\to\C$ is real analytic. Then $\bI_a$ is $s$-injective.
\end{Theorem}

This result is based on the complex stationary phase method of Sj\"ostrand~\cite{sjostrand1982singularites}, which was already used in \cite{frigyik2008x,holman2013generic,holman2010doppler,stefanov2008nonsimple,zhou2017generic}. Theorem~\ref{thm::main 1} can be considered as a generalization of the corresponding result in \cite{stefanov2005boundary}. For $\cI_a$, analogous results are given in \cite{frigyik2008x,holman2010doppler,zhou2017generic}.

We also give a stability estimate for $\bI_a$ in terms of its normal operator following \cite{stefanov2008sharp,stefanov2004stability}. To state this result, let us embed $M$ into the interior of a compact manifolds $\widetilde M$ with boundary and extend the metric $g$ to $\widetilde M$ and keep the same notation for the extension, choosing $(\widetilde M,g)$ to be sufficiently close to $(M,g)$ so that it remains simple. We also extend the attenuation coefficient $a$ to $\widetilde M$ smoothly and keep the same notation for the extension.

We denote by $\tbI_a$ the attenuated geodesic ray transform on $\widetilde M$. Then the \emph{\bfseries normal operator} is defined as $\tbN_{a}:=(\tbI_a)^*\tbI_a$. Let $\mathcal E_{\widetilde M}$ be the operator which extends all pairs on $M$ to $\widetilde M\setminus M$ by zero. In this way, we can and shall consider $\tbI_a$ and $\tbN_{a}$ acting on pairs on $M$ as $\tbI_a:=\tbI_a \mathcal E_{\widetilde M}$ and $\tbN_{a}:=\tbN_{a}\mathcal E_{\widetilde M}$. As it was pointed out in \cite[Section~2]{zhou2017generic}, the knowledge of $\bI_a$ is equivalent to that of~$\tbI_a$.

We show in Section~\ref{sctn::space of pairs} that every $[f,\alpha]\in L^2(M;S^{2}_M{\times} \Lambda^{1}_M)$ can be uniquely decomposed as
$$
[f,\alpha]=[h,\beta]+d_a[w,\phi],
$$
with $[h,\beta]\in L^2(M;S^{2}_M{\times} \Lambda^{1}_M)$ and $[w,\phi]\in H^{1}_0(M;\Lambda^{1}_M{\times} \C)$ such that $\delta_a[h,\beta]=0$. Here and in what follows, $-\delta_a$ is the formal adjoint of $d_a$ under the $L^2$-inner product. We also write $\cS_{a}[f,\alpha]:=[h,\beta]$, which turns out to be a bounded operator $L^2(M;S^2_M\times \Lambda^1_M)\to L^2(M;S^2_M\times \Lambda^1_M)$.

Our second result is on stability estimates for $\bI_a$ in terms of the normal operator $\tbN_{a}$.

\begin{Theorem}\label{thm::main 2}
Let $(M,g)$ be a simple manifold and let $a\in C^\infty(M;\C)$. Suppose $\bI_a$ is $s$-injective.

\begin{itemize}
\item[(a)] There is a constant $C>0$ such that
\begin{equation}\label{main stability estimate}
\|\cS_a[f,\alpha]\|_{L^2(M;S^2_M\times \Lambda^1_M)}/C\le\|\widetilde\bN_a[f,\alpha]\|_{H^1(\widetilde M;S^2_{\widetilde M}\times \Lambda^1_{\widetilde M})}\le C\|\cS_a[f,\alpha]\|_{L^2(M;S^2_M\times \Lambda^1_M)}
\end{equation}
for all $[f,\alpha]\in L^2(M;S^2_M\times \Lambda^1_M)$.\smallskip

\item[(b)] There is $\varepsilon>0$ so that if $(g,a)$ is replaced by $(\tilde g,\tilde a)$ satisfying $\|g-\tilde g\|_{C^3(\widetilde M; S^2_{\widetilde M})}\le \varepsilon$ and $\|a-\tilde a\|_{C^3(\widetilde M; \C)}\le \varepsilon$, the estimate \eqref{main stability estimate} remains true. Moreover, the constant $C>0$ is uniform, depending only on $(g,a)$.
\end{itemize}
\end{Theorem}

This generalizes stability estimates which were proven in \cite{stefanov2005boundary}. Similar results were obtained in \cite{frigyik2008x,holman2010doppler,salo2011attenuated,zhou2017generic} for $\cI_a$.

Combining Theorem~\ref{thm::main 1} and Theorem~\ref{thm::main 2}, one can see that $\bI_a$ is $s$-injective for a generic set of simple metrics and attenuations.

\begin{Corollary}
There exists an open dense set of $(g,a)$ with $(M,g)$ simple so that $\bI_a$ is $s$-injective and \eqref{main stability estimate} holds.
\end{Corollary}

The arguments of this paper also apply to generalize the presented results in several directions:
\begin{itemize}
\item For other types of attenuations including matrix-valued ones and linearly-dependent on direction. Such attenuations play an important role in differential geometry and physics (linear connections and Higgs fields); see \cite{guillarmou2016negconnections,paternain2012attenuated,paternain2016geodesic,zhou2017generic} and references therein. In the current paper, we restrict our attention just to scalar-valued attenuations which depend on position only, since this case appears most in applications~\cite{budinger1979emission,chung2017hodge,ferreira2009limiting,juhlin1992principles}. %One can also prove similar results for 

\item For $I_a$ acting on tensor fields of any rank, after some minor adjustments in the proofs; see Remark~\ref{rmk::Korn}. For ease of notation and readability, we have limited ourselves to $\bI_a$.

\item For a class of non-simple compact manifolds as in \cite{frigyik2008x,stefanov2008nonsimple}. Such a manifold allows conjugate points and trapped geodesics, and have boundary which is not necessarily convex. The integration is then taken over non-trapped geodesics only, so the given data is incomplete. More precisely, the assumption is that the union of the conormal bundles of nontrapping geodesics without conjugate points cover $T^*M$. If $n=2$, this condition guarantees the absence of conjugate points but not the absence of trapped geodesics.
\end{itemize}

Finally, we mention the recent breakthrough in \cite{uhlmann2016inverse}, where the local injectivity of $I_0^0$ was proved near a point $p\in\p M$, provided that $n\ge 3$ and $\p M$ is strictly convex near $p$. Their approach is based on the scattering calculus of Melrose \cite{melrose1994spectral} and the requirement $n\ge 3$ is needed to guarantee ellipticity of the normal operator near $p$. This result was further used to prove the global injectivity of $I_0^0$ when the manifold $(M,g)$ is globally foliated by strictly convex hypersurfaces; see \cite{uhlmann2016inverse}. This method was later adjusted to prove analogous local and global results for $I_0^1$ and $I_0^2$ in \cite{stefanov2014inverting}, and for $\cI_a$ in \cite{paternain2016geodesic}. It is likely that the approach used in these papers could be extended to $\bI_a$ with some modifications. We reserve this for future work.

This paper is organized as follows. In Section~\ref{sctn::space of pairs}, we present some important notions and properties of the space of pairs [$2$-tensor, $1$-form]. In Section~\ref{sctn::normal operator}, we study the normal operator $\tbN_a$. We show that it is a pseudodifferential operator of order $-1$, which turns out to be elliptic on pairs $[f,\alpha]$ with $\delta_a[f,\alpha]=0$. Using the ellipticity, we then construct a parametrix for $\tbN_a$. Section~\ref{sctn::stability result} contains the proof of Theorem~\ref{thm::main 2}. Finally, we prove Theorem~\ref{thm::main 1} in Section~\ref{sctn::injectivity result}.

\subsection*{Acknowledgements} The author is grateful to Professor Plamen Stefanov for his suggestions on an earlier version of this paper. The work was partially supported by AMS-Simons travel grant.

\section{The spaces of pairs}\label{sctn::space of pairs}
In what follows, we use the same notation for a pair $[f,\alpha]\in C^\infty(M;S^2_M\times\Lambda^1_M)$ and the induced function $[f,\alpha](x,v):=f_{ij}(x)\,v^i v^j+\alpha_j(x)\,v^j$ on $SM$ leaving it clear from the context when we mean $[f,\alpha]$ to induce a function on $SM$.

The inner products in the spaces $L^2(M;S^2_M\times \Lambda^1_M)$ and $L^2(M;\Lambda^1_M\times\C)$ are given by
\begin{align*}
([f,\alpha],[h,\beta])_{L^2(M;S^2_M\times \Lambda^1_M)}&=\int_M \<f,\overline{h}\>_g+\<\alpha,\overline{\beta}\>_g\,d\Vol_g,\quad [f,\alpha],[h,\beta]\in L^2(M;S^2_M\times \Lambda^1_M),&\\
([u,\varphi],[w,\phi])_{L^2(M;\Lambda^1_M\times\C)}&=\int_M \<u,\overline{w}\>_g+\varphi\overline \phi\,d\Vol_g,\qquad\,\, [u,\varphi],[w,\phi]\in L^2(M;\Lambda^1_M\times\C),&
\end{align*}
where $d\Vol_g$ is the volume form on $(M,g)$. Assume that $a\in C^\infty(M,\C)$. Consider the following operators $d_a:H^1(M;\Lambda^1_M\times\C)\to L^2(M;S^2_M\times \Lambda^1_M)$ and $\delta_a: H^1(M;S^2_M\times \Lambda^1_M)\to L^2(M;\Lambda^1_M\times\C)$ defined by
\begin{align*}
d_a[w,\phi]&=[d^sw+a\phi g,d\phi+aw],\,\,\quad [w,\phi]\in H^1(M;\Lambda^1_M\times\C),\\
\delta_a[f,\alpha]&=[\delta f-\overline a\alpha,\delta \alpha-\overline a\tr(f)],\quad [f,\alpha]\in H^1(M;S^2_M\times \Lambda^1_M),
\end{align*}
where $\tr(f)=g^{ij}f_{ij}$ in local coordinates. The following integration by parts formula holds for these operators
$$
(\delta_a[f,\alpha],[w,\phi])_{L^2(M;\Lambda^1_M\times\C)}+([f,\alpha],d_a[w,\phi])_{L^2(M;S^2_M\times \Lambda^1_M)}=-\int_{\p M}\<j_\nu f,\overline w\>_g+j_\nu\alpha\,\overline\phi\,d\sigma_{\p M}
$$
where $d\sigma_{\p M}$ is the volume form on the boundary $\p M$ induced by $d\Vol_g$, and $j_\nu f:=(\nu^jf_{ij})$ and $j_\nu \alpha:=\nu^j\alpha_j$ in local coordinates. In particular, $d_a^*=-\delta_a$.

\begin{Proposition}\label{decomposition of pairs}
Let $a\in C^\infty(M;\C)$. For a given $[f,\alpha]\in L^2(M;S^2_M\times \Lambda^1_M)$ there are unique $[h,\beta]\in L^2(M;S^2_M\times \Lambda^1_M)$ and $[w,\phi]\in H^{1}_0(M;\Lambda^1_M\times \C)$ such that
$$
[f,\alpha]=[h,\beta]+d_a[w,\phi]\quad\text{and}\quad \delta_a[h,\beta]=0.
$$
Moreover, the projection operators $\cS_{a}[f,\alpha]:=[h,\beta]$ and $\cP_{a}[f,\alpha]:=d_a[w,\phi]$ are bounded from $L^2(M;S^2_M\times \Lambda^1_M)\to L^2(M;S^2_M\times \Lambda^1_M)$.
\end{Proposition}
For the proof we need the following result.

\begin{Proposition}\label{solution operator for Dirichlet problem}
Introduce the operator $\Delta_{g,a}:=\delta_a d_a:H^1_0(M;\Lambda^1_M\times \C)\to H^{-1}(M;\Lambda^1_M\times \C)$. %For a given $\cf\in \cH^{-1}(M,\C)$, the Dirichlet problem
%\begin{equation}\label{Dirichlet problem for elliptic pde}
%-\Delta_{g,a} \cu = \cf\quad\text{in}\quad M,\qquad\cu|_{\p M}=0% \in \cH^{k-1}(M,\C)
%\end{equation}
%has a unique solution $\cu\in \cH^1(M,\C)$, and there is a constant $C>0$ such that the following estimate holds
%$$
%\|\cu\|_{\cH^1(M,\C)}\le C \|\cf\|_{\cH^{-1}(M,\C)}.
%$$
%In other words, t
There is a bounded solution operator
\begin{equation}\label{stability for solution of elliptic bvp}
(-\Delta^D_{g,a})^{-1}:H^{-1}(M;\Lambda^1_M\times \C)\to H^1_0(M;\Lambda^1_M\times \C)
\end{equation}
such that $(-\Delta_{g,a})(-\Delta^D_{g,a})^{-1}=\id$.
\end{Proposition}
\begin{proof}
One can see that $\sigma_p(-\Delta_{g,a})(x,\xi)=|\xi|^2$, so $-\Delta_{g,a}$ is a second order elliptic operator. Also, the Dirichlet boundary condition is coercive. Therefore, it is left to show that this elliptic problem has trivial kernel and cokernel.

We first prove the triviality of the kernel. Suppose $[w,\phi]\in H^1_0(M;\Lambda^1_M\times \C)$ with $-\Delta_{g,a}[w,\phi]=0$ in $M$. Then $[w,\phi]\in C^\infty(M;\Lambda^1_M\times \C)$ by ellipticity. One can also check that
$$
\|d_a[w,\phi]\|^2_{L^2(M;S^2_M\times \Lambda^1_M)}=(-\Delta_{g,a}[w,\phi],[w,\phi])_{L^2(M;\Lambda^1_M\times \C)}=0.
$$
Hence, we have $d_a[w,\phi]=0$ in $M$. For any $x_0\in M^{\rm int}$ and any $v_0\in S_{x_0}M$, there is a unique geodesic $\gamma_{x_0,v_0}$ such that $x_0=\gamma_{x_0,v_0}(0)$ and $v_0=\dot\gamma_{x_0,v_0}(0)$. Let us also write $x_1=\gamma_{x_0,v_0}(\tau(x,v_0))$ and $v_1=\dot\gamma_{x_0,v_0}(\tau(x,v_0))$. Then clearly $(x_1,v_1)\in\p_-SM$. Since
$$
X\big(U^{-1}_a(x,v)[w,\phi](x,v)\big)=U^{-1}_a(x,v)\,d_a[w,\phi](x,v)=0,\quad (x,v)\in SM,
$$
the expression $U^{-1}_a(x,v)[w,\phi](x,v)$ is constant along the geodesic $\gamma_{x_0,v_0}$. Therefore,
$$
U^{-1}_a(x_0,v_0)[w,\phi](x_0,v_0)=U^{-1}_a(x_1,v_1)[w,\phi](x_1,v_1).
$$
According to the hypothesis $[w,\phi]|_{\p M}=0$, this implies that $[w,\phi](x_0,v_0)$. Since $(x_0,v_0)\in SM^{\rm int}$ was arbitrary, we can conclude that $[w,\phi]=0$ in $M$.

To prove the triviality of the cokernel, consider $[u,\varphi]\in C^\infty(M;\Lambda^1_M\times\C)$ such that
$$
(-\Delta_{g,a}[w,\phi],[u,\varphi])_{L^2(M;\Lambda^1_M\times\C)}=0\quad\text{for all}\quad[w,\phi]\in C^\infty(M;\Lambda^1_M\times\C)\quad\text{with}\quad[w,\phi]|_{\p M}=0.
$$
Then for all $[w,\phi]\in C^\infty_0(M^{\rm int};\Lambda^1_M\times\C)$,
\begin{align*}
0&=(-\Delta_{g,a}[w,\phi],[u,\varphi])_{L^2(M;\Lambda^1_M\times\C)}=(d_a[w,\phi],d_a[u,\varphi])_{L^2(M;S^2_M\times \Lambda^1_M)}\\
&=([w,\phi],-\Delta_{g,a}[u,\varphi])_{L^2(M;\Lambda^1_M\times\C)}.
\end{align*}
%Therefore, $\delta_a d_a\mathbf w$ is orthogonal to any $\mathbf u\in\cC^\infty_0(M^{\rm int},\C)$. 
This implies that $-\Delta_{g,a}[u,\varphi]=0$. For arbitrary $[v,\psi]\in C^\infty(M;\Lambda^1_M\times\C)$ with $[v,\psi]|_{\p M}=0$,
\begin{align*}
0&=([v,\psi],-\Delta_{g,a}[u,\varphi])_{L^2(M,\Lambda^1_M\times\C)}=(d_a[v,\psi],d_a[u,\varphi])_{L^2(M;S^2_M\times \Lambda^1_M)}\\
&=\int_{\p M}\<j_\nu d^s v+a\psi\nu,\overline u\>_g+(j_\nu d\psi+a j_\nu v)\overline\varphi\,d\sigma_{\p M}.
\end{align*}
In the last step we used the fact that $[u,\varphi]$ is in the cokernel. Then we get $[u,\varphi]|_{\p M}=0$, which allows us to conclude that $[u,\varphi]=0$ since $-\Delta_{g,a}[u,\varphi]=0$.

To prove the boundedness of $(-\Delta_{g,a}^D)^{-1}$, given $[w,\phi]\in H^1_0(M;\Lambda^1_M\times\C)$, using Korn's inequality \cite[Corollary~5.12.3]{taylor2011partial} in combination with Poincar\'e type inequality and integration by parts, we get
\begin{multline*}
\|[w,\phi]\|_{H^1(M;\Lambda^1_M\times\C)}^2\\
\le C\big(\|d_a[w,\phi]\|_{L^2(M;S^2_M\times \Lambda^1_M)}^2+\|[w,\phi]\|_{L^2(M;\Lambda^1_M\times\C)^2}^2\big)\le C\|d_a[w,\phi]\|_{L^2(M;S^2_M\times \Lambda^1_M)}^2\\
=C\<-\Delta_{g,a}[w,\phi],[w,\phi]\>\le C\|-\Delta_{g,a}[w,\phi]\|_{H^{-1}(M;\Lambda^1_M\times\C)}\|[w,\phi]\|_{H^1(M;\Lambda^1_M\times\C))},
\end{multline*}
where $\<\cdot,\cdot\>$ is the duality between $H_0^1(M;\Lambda^1_M\times\C)$ and $H^{-1}(M;\Lambda^1_M\times\C)$. This finishes the proof.% Therefore, $\|u\|_{\cH^1(M,\C)}\le C\|\cf\|_{\cH^{-1}(M,\C)}$ as desired.
\end{proof}
\begin{Remark}\label{rmk::Korn}{\rm
To prove analogs of Proposition~\ref{decomposition of pairs} and Theorem~\ref{thm::main 2} for higher ranked tensors, one needs to derive a generalization of Korn's inequality. This can be achieved following the same reasonings as in \cite{duvaut1976inequalities,mclean2000strongly} and using \cite[Lemma~7.2]{dairbekov2011conformal}.
}\end{Remark}

Now we can prove Proposition~\ref{decomposition of pairs}.

\begin{proof}[Proof of Proposition~\ref{decomposition of pairs}]
Use Proposition~\ref{solution operator for Dirichlet problem} by setting $[u,\varphi]:=-\delta_a[f,\alpha]$, $[w,\phi]:=(-\Delta_{g,a}^D)^{-1}[u,\varphi]$ and $[h,\beta]:=[f,\alpha]-d_a[w,\phi]$. Then $\cP_a=-d_a(-\Delta_{g,a}^D)^{-1}\delta_a$ and $\cS_a=\id-\cP_a$ are bounded from $L^2(M;S^2_M\times \Lambda^1_M)\to L^2(M;S^2_M\times \Lambda^1_M)$.
\end{proof}

Later, we will also need the following result.

\begin{Proposition}\label{solvability of elliptic bvp}
Let $[v,\varphi]\in H^{-1}(M;\Lambda^1_M\times \C)$ and $[w_0,\phi_0]\in H^{1/2}(\p M;\Lambda^1_M\times \C)$. Then the boundary value problem
\begin{equation}\label{bvp for elliptic pde}
-\Delta_{g,a}[w,\phi] = [v,\varphi]\quad\text{in}\quad M,\qquad[w,\phi]|_{\p M}=[w_0,\phi_0]% \in \cH^{k-1}(M,\C)
\end{equation}
has a unique solution $[w,\phi]\in H^1(M;\Lambda^1_M\times\C)$, and there is a constant $C>0$ such that the following estimate holds
\begin{equation}\label{stability for solution of elliptic bvp}
\|[w,\phi]\|_{H^1(M;\Lambda^1_M\times \C)}\le C\big(\|[v,\varphi]\|_{H^{-1}(M;\Lambda^1_M\times\C)}+\|[w_0,\phi_0]\|_{H^{1/2}(\p M;\Lambda^1_M\times\C)}\big).
\end{equation}
\end{Proposition}
\begin{proof}
Since $[w_0,\phi_0]\in H^{1/2}(\p M;\Lambda^1_M\times\C)$, there is $[\tilde w,\tilde \phi]\in H^1(M;\Lambda^1_M\times\C)$ such that $[\tilde w,\tilde\phi]|_{\p M}=[w_0,\phi_0]$ and
\begin{equation}\label{extension is bounded by trace}
\|[\tilde w,\tilde \phi]\|_{H^1(M;\Lambda^1_M\times\C)}\le C\|[w_0,\phi_0]\|_{H^{1/2}(\p M;\Lambda^1_M\times \C)}.
\end{equation}
Set $[w_1,\phi_1]:=[w,\phi]-[\tilde w,\tilde\phi]$. Then \eqref{bvp for elliptic pde} is equivalent to
$$
-\Delta_{g,a} [w_1,\phi_1]=[\tilde v,\tilde\varphi]\quad\text{in}\quad M,\qquad [\tilde v,\tilde\varphi]|_{\p M}=0,
$$
where $[\tilde v,\tilde\varphi]:=[v,\varphi]+\Delta_{g,a}[\tilde w,\tilde\phi]\in H^{-1}(M;\Lambda^1_M\times\C)$. By Proposition~\ref{solution operator for Dirichlet problem}, there is a unique $[w_1,\phi_1]\in H^1_0(M;\Lambda^1_M\times\C)$ solving $-\Delta_{g,a} [w_1,\phi_1]=[\tilde v,\tilde\varphi]$ and such that $\|[w_1,\phi_1]\|_{H^1(M;\Lambda^1_M\times\C)}\le C\|[\tilde v,\tilde\varphi]\|_{H^{-1}(M;\Lambda^1_M\times\C)}$. Using triangle inequalities, this implies
$$
\|[w,\phi]\|_{H^1(M;\Lambda^1_M\times\C)} \le C\|[v,\varphi]\|_{H^{-1}(M;\Lambda^1_M\times\C)}+C\|\Delta_{g,a}[\tilde w,\tilde\phi]\|_{H^{-1}(M;\Lambda^1_M\times\C)}+\|[\tilde w,\tilde\phi]\|_{H^1(M;\Lambda^1_M\times\C)}.
$$
By boundedness of $\Delta_{g,a}:H^1_0(M;\Lambda^1_M\times\C)\to H^{-1}(M;\Lambda^1_M\times\C)$ and \eqref{extension is bounded by trace}, we come to \eqref{stability for solution of elliptic bvp} as desired.
\end{proof}

\section{The adjoint and normal operators}\label{sctn::normal operator}

\subsection{Transport equations and $\bI_a$}
The transform $\bI_a[f,\alpha]$ can be realized as the trace on $\p_+ SM$ of the solution $u:SM\to \C$ to the following transport equation on $SM$,
$$
Xu+au=-[f,\alpha]\quad\text{in}\quad SM,\qquad u|_{\p_-SM}=0.
$$
This equation has a unique solution $u$, since on any fixed geodesic the transport equation is an ODE with zero initial condition and an integral expression gives us that $u|_{\partial_+ SM}$ matches $\bI_a[f,\alpha]$.

For $w\in C^\infty(\p_+SM,\C^n)$ given, let us denote $w_\psi (x,v) := w(\gamma_{x,v}(-\tau(x,-v)),\dot\gamma_{x,v}(-\tau(x,-v)))$ the unique solution $u$ to the transport problem
$$
Xu=0\quad\text{in}\quad SM,\qquad u\big|_{\p_+SM}=w.
$$
Define the integrating factor $U_{a}:SM\to\C$, unique solution to 
$$
(X+a)U_a=0\quad\text{in}\quad SM,\qquad U_{a}|_{\p_+SM}=1,
$$
whose integral expression is given by
$$
U_{a}(x,v)=\exp\bigg(-\int^0_{-\tau(x,-v)}a(\gamma_{x,v}(s))\,ds\bigg),\quad (x,v)\in SM.
$$
By solving explicitly the transport equation along the geodesic, one can show that
$$
U_{a}(\gamma_{x,v}(t),\dot\gamma_{x,v}(t))=\exp\bigg({-\int_0^t a(\gamma_{x,v}(s))\,ds}\bigg),\quad (x,v)\in SM,
$$
and hence the following integral formula holds
$$
\bI_{a} [f,\alpha](x,v)=\int_0^{\tau(x,v)}U_{a}^{-1}(\gamma_{x,v}(t),\dot\gamma_{x,v}(t))[f,\alpha]\big(\gamma_{x,v}(t),\dot\gamma_{x,v}(t)\big)\,dt,\quad (x,v)\in\p_+SM.
$$

\subsection{The adjoint of $\bI_a$}\label{sec:adjoints}

Denote by $L_\mu^2(\p_+ SM;\C)$ the completion of $C_c^\infty(\p_+ SM;\C)$ for the inner product
\begin{align*}
(w,w')_{L^2_\mu(\p_+SM;\C)}=\int_{\p_+ SM}w\overline{w'}\,d\mu, \qquad d\mu(x,v):=\< v,\nu_x\>_{g(x)} d\Sigma^{2n-2}(x,v),
\end{align*}
where $d\Sigma^{2n-2}$ be the volume form on $\p(SM)$. Using Santal\'o formula \cite[Lemma~A.8]{dairbekov2007boundary}, one can show that $\bI_a$ can be extended to a bounded operator $\bI_a:L^2(M;S^2_M\times \Lambda^1_M)\to L_\mu^2(\p_+ SM;\C)$.

Now, consider the adjoint $\bI_a^*:L^2_\mu(\p_+SM;\C)\to L^2(M;S^2_M\times \Lambda^1_M)$ of $\bI_{a}$. It was shown in \cite{assylbekov2017inversion} that
$$
(\bI_a[f,\alpha],w)_{L^2_\mu(\p_+SM,\C)}=\int_{SM}[f,\alpha](x,v)\,\overline{U_{-\overline a}(x,v)w_\psi(x,v)}\,d\Sigma^{2n-1}(x,v),
$$
for $w\in L^2_\mu(\p_+SM,\C)$. From this, one can get the following explicit expression for the adjoint of $\bI_a$
$$
\bI_{a}^*w=[(I_a^2)^*w,(I_a^1)^*w]=\Big[\int_{S_x M}v^{i} v^{j} U_{-\overline a}(x,v) w_\psi(x,v)\,d\sigma_x(v),\, \int_{S_x M}v^{i}U_{-\overline a}(x,v) w_\psi(x,v)\,d\sigma_x(v)\Big],
$$
where $d\sigma_x$ is the measure on $S_xM$.

\subsection{The normal operator}\label{section::normal operator} We embed $M$ into the interior of a compact manifolds $\widetilde M$ with boundary and extend the metric $g$ to $\widetilde M$ and keep the same notation for the extension, choosing $(\widetilde M,g)$ to be sufficiently close to $(M,g)$ so that it remains simple. We also extend the attenuation coefficient $a$ to $\widetilde M$ smoothly and keep the same notation for the extension.

We denote by $\tbI_a$ the attenuated geodesic ray transform on $\widetilde M$. The the \emph{\bfseries normal operator} is defined as $\tbN_{a}:=\tbI_a^*\tbI_a$. We say that $\tbN_{a}$ is \emph{\bfseries elliptic on $a$-solenoidal pairs}, if $\diag(d_a\Lambda \delta_a,\tbN_{a})$, acting on pairs, is elliptic (as a system of pseudodifferential operators of order $-1$), where $\Lambda$ is a proper pseudodifferential operator on $\widetilde M^{\rm int}$ with principal symbol $1/|\xi|^3$. Recall that $\diag(d_a\Lambda \delta_a,\tbN_{a})$ is an elliptic system if $\det\sigma_p(\diag(d_a\Lambda \delta_a,\tbN_{a}))(x,\xi)\neq 0$ for $(x,\xi)\in TM^{\rm int}\setminus \{0\}$; see \cite[page~46]{shubin2001pseudodifferential}.

\begin{Proposition}\label{prop::ellipticity of N}
$\tbN_{a}$ is a pseudodifferential operator of order $-1$ in $\widetilde M^{\rm int}$ which is elliptic on $a$-solenoidal pairs.
\end{Proposition}
\begin{proof}
First, we prove that $\tbN_a$ is a pseudodifferential operator of order $-1$ in $M^{\rm int}$. Recall that $\tbN_a:L^2(M;S^2_M\times \Lambda^1_M)\to L^2(M;S^2_M\times \Lambda^1_M)$. Therefore, we introduce the following notation
$$
\tbN_{a}[f,\alpha]=[\,\tbN_{a}^{22}f+\tbN_{a}^{21}\alpha\,,\,\tbN_{a}^{12}f+\tbN_{a}^{11}\alpha\,],\quad [f,\alpha]\in L^2(M;S^2_M\times \Lambda^1_M),
$$
where
$$
\tbN_{a}^{22}:=(I_{a}^2)^*I_{a}^{2},\quad \tbN_{a}^{21}:=(I_{a}^2)^*I_{a}^{1},\quad \tbN_{a}^{12}:=(I_{a}^1)^*I_{a}^{2},\quad \tbN_{a}^{11}:=(I_{a}^1)^*I_{a}^{1}.
$$
Then one can show that
\begin{align*}
(\tbN_{a}^{22}f){}^{i'j'}(x)&=\int_{S_x \widetilde M}v^{i'} v^{j'} \widetilde U_{-\overline a}(x,v)\int_{-\tilde\tau(x,-v)}^{\tilde\tau(x,v)}\widetilde U_{a}^{-1}(\gamma_{x,v}(t),\dot\gamma_{x,v}(t))f_{ij}(\gamma_{x,v}(t))\dot\gamma_{x,v}^i(t)\dot\gamma_{x,v}^j(t)\,dt\,d\sigma_x(v),\\
(\tbN_{a}^{21}\alpha){}^{i'j'}(x)&=\int_{S_x \widetilde M}v^{i'}v^{j'} \widetilde U_{-\overline a}(x,v)\int_{-\tilde\tau(x,-v)}^{\tilde\tau(x,v)}\widetilde U_{a}^{-1}(\gamma_{x,v}(t),\dot\gamma_{x,v}(t))\alpha_{i}(\gamma_{x,v}(t))\dot\gamma_{x,v}^i(t)\,dt\,d\sigma_x(v),\\
(\tbN_{a}^{12}f){}^{i'}(x)&=\int_{S_x \widetilde M}v^{i'} \widetilde U_{-\overline a}(x,v)\int_{-\tilde\tau(x,-v)}^{\tilde\tau(x,v)}\widetilde U_{a}^{-1}(\gamma_{x,v}(t),\dot\gamma_{x,v}(t))f_{ij}(\gamma_{x,v}(t))\dot\gamma_{x,v}^i(t)\dot\gamma_{x,v}^j(t)\,dt\,d\sigma_x(v),\\
(\tbN_{a}^{11}\alpha){}^{i'}(x)&=\int_{S_x \widetilde M}v^{i'} \widetilde U_{-\overline a}(x,v)\int_{-\tilde\tau(x,-v)}^{\tilde\tau(x,v)}\widetilde U_{a}^{-1}(\gamma_{x,v}(t),\dot\gamma_{x,v}(t))\alpha_{i}(\gamma_{x,v}(t))\dot\gamma_{x,v}^i(t)\,dt\,d\sigma_x(v).
\end{align*}
Following \cite{frigyik2008x,holman2010doppler}, we use \cite[Lemma~B.1]{dairbekov2007boundary} to deduce that $\tbN_{a}$ is a pseudodifferential operator of order $-1$, and the principal symbols of the above operators are as follows:
\begin{align*}
\sigma_p(\tbN_{a}^{22})^{i'j'ij}(x,\xi)&=2\pi\int_{S_x \widetilde M}\omega^{i'}\omega^{j'}\omega^{i}\omega^{j}\delta(\<\omega,\xi\>_{g(x)})\widetilde U_{-2\Re(a)}(x,\omega)\,d\sigma_x(\omega),\\
\sigma_p(\tbN_{a}^{21})^{i'j'i}(x,\xi)&=2\pi\int_{S_x \widetilde M}\omega^{i'}\omega^{j'}\omega^{i}\delta(\<\omega,\xi\>_{g(x)})\widetilde U_{-2\Re(a)}(x,\omega)\,d\sigma_x(\omega),\\
\sigma_p(\tbN_{a}^{12})^{i'ij}(x,\xi)&=2\pi\int_{S_x \widetilde M}\omega^{i'}\omega^{i}\omega^{j}\delta(\<\omega,\xi\>_{g(x)})\widetilde U_{-2\Re(a)}(x,\omega)\,d\sigma_x(\omega),\\
\sigma_p(\tbN_{a}^{11})^{i'i}(x,\xi)&=2\pi\int_{S_x \widetilde M}\omega^{i'}\omega^{i}\delta(\<\omega,\xi\>_{g(x)})\widetilde U_{-2\Re(a)}(x,\omega)\,d\sigma_x(\omega).
\end{align*}
Now, we prove ellipticity. For this, note that the ellipticity of $\diag(d_a\Lambda \delta_a,\tbN_{a})$ is equivalent to saying that the principal symbol $\sigma_p(\diag(d_a\Lambda \delta_a,\tbN_{a}))(x,\xi)$, acting on pairs, is injective for every $(x,\xi)\in T\widetilde M^{\rm int}\setminus \{0\}$; see the comments preceding \cite[Definition~7.1]{treves1980introduction}. Assume that, for a constant symmetric $2$-tensor $f$ and a $1$-form $\alpha$, $\sigma_p(\tbN_{a})[f,\alpha]=0$ and $\sigma_p(d_a\Lambda \delta_a)[f,\alpha]=0$ at some $(x,\xi)\in T\widetilde M^{\rm int}\setminus \{0\}$. Then it follows that
\begin{equation}\label{f and h at xi are zero}
f_{ij}\xi^{i}=0,\quad \alpha_{j}\xi^{j}=0
\end{equation}
and
$$
0=\<\sigma_p(\tbN_{a})[f,\alpha],[f,\alpha]\>_{g(x)}=2\pi\int_{S_x\widetilde M}|f_{ij}\omega^i \omega^j+\alpha_{j}\omega^j|^2\delta(\<\omega,\xi\>_{g(x)})\widetilde U_{-2\Re(a)}(x,\omega)\,d\sigma_x(\omega),
$$
where the inner product $\<\cdot,\cdot\>_g$ is for pairs. Note that $\widetilde U_{-2\Re(a)}>0$ and that the set $S_{x,\xi}:=\{\omega\in S_x\widetilde M:\<\omega,\xi\>_{g(x)}=0\}$ is non-empty. Therefore, for all such $\omega$, we get $f_{ij}\omega^i \omega^j+\alpha_{i}\omega^i=0$. Since $-\omega$ is also in $S_{x,\xi}$, we also have $f_{ij}\omega^i \omega^j-\alpha_{i}\omega^i=0$. These two equalities imply that $f_{ij}\omega^i \omega^j=0$ and $\alpha_{i}\omega^i=0$ for all $\omega\in S_{x,\xi}$. Combining these with \eqref{f and h at xi are zero}, we conclude that $f=0$ and $\alpha=0$. Thus, $\tbN_{a}$ is elliptic on $a$-solenoidal pairs.
\end{proof}

Let $\mathcal E_{\widetilde M}:L^2(M;S^2_M\times \Lambda^1_M)\to L^2(\widetilde M;S^2_{\widetilde M}\times \Lambda^1_{\widetilde M})$ be the operator which extends all pairs in $L^2(M;S^2_M\times \Lambda^1_M)$ to $\widetilde M\setminus M$ by zero. In this way, we consider $L^2(M;S^2_M\times \Lambda^1_M)$ as a subspace of $L^2(\widetilde M;S^2_{\widetilde M}\times \Lambda^1_{\widetilde M})$. As it was pointed out in \cite[Section~2]{zhou2017generic}, for $[f,\alpha]\in L^2(M;S^2_M\times \Lambda^1_M)$, the knowledge of $\bI_a[f,\alpha]$ and $\tbI_a[f,\alpha]:=\tbI_a \mathcal E_{\widetilde M} [f,\alpha]$ is equivalent. We can also consider $\tbN_a:=\tbI_a^*\tbI_a$ as the bounded operator $\tbN_a:L^2(M;S^2_M\times \Lambda^1_M)\to L^2(\widetilde M;S^2_{\widetilde M}\times \Lambda^1_{\widetilde M})$.

\section{Generic stability}\label{sctn::stability result}

In this section we prove Theorem~\ref{thm::main 2}. Since $\diag(d_a\Lambda\delta_a, \tbN_{a})$, acting on pairs, is elliptic, there are pseudodifferential operators $Q$ and $Q'$ in $\widetilde M^{\rm int}$ of order $1$ such that for all $[f,\alpha]\in L^2(\widetilde M;S^2_{\widetilde M}\times\Lambda^1_{\widetilde M})\cap \cE'(\widetilde M^{\rm int};S^2_{\widetilde M}\times \Lambda^1_{\widetilde M})$,
\begin{equation}\label{QN+Q'C=1}
Q\tbN_{a}[f,\alpha]+Q'd_a\Lambda\delta_a[f,\alpha]=[f,\alpha]+K_1[f,\alpha]\quad\text{in}\quad \widetilde M^{\rm int},
\end{equation}
where $K_1$ is a smoothing operator acting on pairs in $\cE'(\widetilde M^{\rm int};S^2_{\widetilde M}\times \Lambda^1_{\widetilde M})$. Note that the kernel of $K_1$ may have singularities at $\p \widetilde M$.

Consider a compact smooth manifold $M_1$ such that $M\Subset M_1^{1/2}\Subset\widetilde M^{\rm int}$. We take $M_1$ sufficiently close to $M$ so that $(M_1,g)$ is simple. Consider a pair $[f,\alpha]\in L^2(M;S^2_M\times\Lambda^1_M)$ for whose extensions $\cE_{M_1}[f,\alpha]$ and $\cE_{\widetilde M}[f,\alpha]$ we have unique decompositions $
\cE_{B}[f,\alpha]=\cS_{a}\cE_{B}[f,\alpha]+d_a[w_{B},\phi_{B}]$ with $[w_{B},\phi_{B}]\in H^1_0(B;S^2_{B}\times \Lambda^1_{B})$ and $\delta_a\cS_{a}\cE_{B}[f,\alpha]=0$, where $B=M_1,\widetilde M$. Then $[w_B,\phi_B]=-(-\Delta_{g,a})_B^{-1}\delta_a(\cE_B[f,\alpha])$ and
$$
\cS_{a}\cE_{B}[f,\alpha]=\cE_{B}[f,\alpha]+d_a(-\Delta_{g,a})_B^{-1}\delta_a(\cE_B[f,\alpha]),\qquad B=M_1,\widetilde M,
$$
where $(-\Delta^D_{g,a})^{-1}_{B}:H^{-1}(B;\Lambda^1_{B}\times\C)\to H^1_0(B,\Lambda^1_{B}\times\C)$ are  bounded right inverses of the operators $-\Delta_{g,a}: H^1_0(B;\Lambda^1_{B}\times\C)\to H^{-1}(B;\Lambda^1_{B}\times\C)$ constructed in Proposition~\ref{solution operator for Dirichlet problem}.

We take a cut-off function $\chi\in C^\infty_0(\widetilde M^{\rm int})$ such that $\chi\equiv 1$ near $M_1$. Since $\supp\delta_a(\chi\cS_a\cE_{\widetilde M}[f,\alpha])\subset \widetilde M\setminus M_1$, by pseudolocal property, $Q'd_a\Lambda\delta_a(\chi\cS_a\cE_{\widetilde M}[f,\alpha])$ is smooth near $M_1$. Hence, by \eqref{QN+Q'C=1} we get
$$
Q\tbN_a\chi\cS_a\cE_{\tM}[f,\alpha]=\cS_a\cE_{\tM}[f,\alpha]+K_2[f,\alpha]\quad\text{in}\quad M_1^{\rm int}
$$
where $K_2:=(K_1-Q'd_a\Lambda\delta_a)\chi(\id+d_a(-\Delta_{g,a})_{\tM}^{-1}\delta_a)\cE_{\tM}$ whose kernel is in $C^\infty(M_1\times M_1^{\rm int})$. Since $\chi\equiv 1$ near $M_1$ and $\supp[f,\alpha]\subseteq M$, we have
$$
Q\tbN_a\chi\cS_a\cE_{\tM}[f,\alpha]=Q\tbN_a[f,\alpha]-Q\tbN_a\chi d_a[w_{\tM},\phi_{\tM}]\quad\text{in}\quad \tM^{\rm int}.
$$
Let $D$ be a parametrix for $-\Delta_{g,a}$ in $\tM^{\rm int}$. Then, $D(-\Delta_{g,a})=\id+K_3$ in $\tM^{\rm int}$ with $K_3$ being a smoothing operator in $\tM^{\rm int}$. Hence, it is not difficult to see that
$$
((-\Delta_{g,a})_{\tM}^{-1}-D)\delta_a(\cE_{\tM}[f,\alpha])=-K_4[f,\alpha]\quad\text{in}\quad\tM^{\rm int},
$$
where $K_4:=K_3(-\Delta_{g,a})_{\tM}^{-1}\delta_a\circ\cE_{\tM}$ which is also smoothing in $\tM^{\rm int}$. Then,
$$
[w_{\tM},\phi_{\tM}]=-D\delta_a(\cE_{\tM}[f,\alpha])+K_4[f,\alpha],
$$
and hence
\begin{align*}
Q\tbN_a\chi d_a[w_{\tM},\phi_{\tM}]&=Q\tbN_a\chi d_a(-D\delta_a(\cE_{\tM}[f,\alpha])+K_4[f,\alpha])\\
&=Q\tbN_a d_a(\chi(-D\delta_a(\cE_{\tM}[f,\alpha])+K_4[f,\alpha]))-Q\tbN_a \sigma_{d\chi}(-D\delta_a(\cE_{\tM}[f,\alpha])+K_4[f,\alpha])\\
&=-Q\tbN_a \sigma_{d\chi}(-D\delta_a(\cE_{\tM}[f,\alpha])+K_4[f,\alpha])\quad\text{in}\quad M_1^{\rm int},
\end{align*}
where $\sigma_{d\chi}$ is the symmetrized tensor product by $d\chi$. Observe that $K_5:=-Q\tbN_a \sigma_{d\chi}(-D\delta_a\circ \cE_{\tM}+K_4)$ is a pseudodifferential operator of order $-1$ in $M_1^{\rm int}$, and hence so is $K_{-1}:=K_5+K_2$. Therefore, we come to
\begin{equation}\label{QN=Id+K_{-1}}
Q\tbN_a[f,\alpha]=\cS_a\cE_{\tM}[f,\alpha]+K_{-1}[f,\alpha]\quad\text{in}\quad M_1^{\rm int}.
\end{equation}
Next, we wish and shall replace $\cS_a\cE_{\tM}[f,\alpha]$ here by $\cS_a\cE_{M_1}[f,\alpha]$. For this, observe that $\cS_a\cE_{M_1}[f,\alpha]=\cS_a\cE_{\tM}[f,\alpha]+d_a[w,\phi]$, where $[w,\phi]:=[w_{\tM},\phi_{\tM}]-[w_{M_1},\phi_{M_1}]\in H^1(M_1;\Lambda^1_{M_1}\times\C)$ satisfying
$$
-\Delta_{g,a}[w,\phi]=0\quad\text{in}\quad M_1^{\rm int},\qquad [w,\phi]|_{\p M_1}=[w_{\tM},\phi_{\tM}]|_{\p M_1}.
$$
Since $\supp(\cE_{\tM}[f,\alpha])\subseteq M$, we have $-\Delta_{g,a}[w_{\tM},\phi_{\tM}]=0$ in $\tM^{\rm int}\setminus M$. Then elliptic regularity guarantees that $[w_{\tM},\phi_{\tM}]$ is smooth in $\tM^{\rm int}\setminus M_1$, and hence $[w_{\tM},\phi_{\tM}]|_{\p M_1}$ is smooth on $\p M_1$. Hence, $K_7[f,\alpha]:=-d_a[w,\phi]_{M_1}$ is a linear operator from $L^2(M;S^2_M\times\Lambda^1_M)$ to $C^\infty(M_1;S^2_{M_1}\times\Lambda^1_{M_1})$, i.e. smoothing on $M_1$. Therefore, we can rewrite \eqref{QN=Id+K_{-1}} as
\begin{equation}\label{QN=Id+K}
Q\tbN_a[f,\alpha]=\mathcal S_a \mathcal E_{M_1}[f,\alpha]+K[f,\alpha]\quad\text{in}\quad M_1^{\rm int},
\end{equation}
where $K:=K_{-1}+K_7$ which is a pseudodifferential operator of order $-1$ in $M_1^{\rm int}$.

Now, we are ready to prove Theorem~\ref{thm::main 2}. To that end we need the following apriori estimate.

\begin{Proposition}\label{prop::apriori estimate for N}
For every $[f,\alpha]\in L^2(M;S^2_M\times \Lambda^1_M)$,
$$
\|\cS_a[f,\alpha]\|_{L^2(M;S^2_M\times \Lambda^1_M)}\le C\big(\|\tbN_a[f,\alpha]\|_{H^1(\widetilde M;S^2_{\widetilde M}\times \Lambda^1_{\widetilde M})}+\|[f,\alpha]\|_{H^{-1}(\tM;S^2_{\tM}\times \Lambda^1_{\tM})}\big).
$$
\end{Proposition}
\begin{proof}
Starting from \eqref{QN=Id+K}, our first goal is to construct $\mathcal S_a[f,\alpha]$ from $\mathcal S_a \mathcal E_{M_1}[f,\alpha]$. We can write
\begin{equation}\label{solenoidal part of extension of f}
\mathcal S_a \mathcal E_{M_1}[f,\alpha]=\mathcal E_{M_1}\mathcal S_a [f,\alpha]-d_a[w,\phi]\quad\text{in}\quad M_1,
\end{equation}
where $[w,\phi]\in H^1(M_1,\Lambda^1_{M_1}\times\C)$ solves
$$
\delta_a d_a[w,\phi]=\delta_a\mathcal E_{\widetilde M}\mathcal S_a [f,\alpha]\quad\text{in}\quad \widetilde M,\quad [w,\phi]|_{\p M_1}=0.
$$
If we recover $[w,\phi]|_{\p M}\in H^{1/2}(\p M;\Lambda^1_M\times\C)$, then we could recover $[w,\phi]$ in $M$ by solving $\delta_a d_a[w,\phi]=0$ in $M$. Hence, we would recover $\mathcal S_a [f,\alpha]$ in $M$ via \eqref{solenoidal part of extension of f}. Therefore our aim is to recover $[w,\phi]|_{\p M}$.

From \eqref{solenoidal part of extension of f} we get
$$
\mathcal S_a \mathcal E_{M_1}[f,\alpha]=-d_a[w,\phi]\quad\text{in}\quad M_1\setminus M.
$$
Thus, we know $d_a[w,\phi]$ in $\widetilde M\setminus M$. Integrating
$$
\frac{d}{dt}\Big(\widetilde U^{-1}_a(\gamma(t),\dot\gamma(t))[w,\phi](\gamma(t),\dot\gamma(t))\Big)=\widetilde U^{-1}_a(\gamma(t),\dot\gamma(t))d_a[w,\phi](\gamma(t),\dot\gamma(t))
$$
along geodesics $\gamma$ in $M_1\setminus M$ connecting points on $\p M$ and $\p M_1$, and using the fact that $[w,\phi]|_{\p M_1}=0$, we recover $[w,\phi]|_{\p M}$.

For $(x,v)\in SM_1$, let $\ell(x,v)$ be the first positive time when $\gamma_{x,v}(\ell(x,v))\in\p M_1$. Then for any $(x,v)\in SM_1$ such that $\gamma_{x,v}([0,\ell(x,v)])\cap M=\varnothing$,
$$
[w,\phi](x,v)=\widetilde U_a(x,v)\int_0^{\ell(x,v)}\widetilde U_a^{-1}(\gamma_{x,v}(t),\dot \gamma_{x,v}(t))\,\mathcal S_a \mathcal E_{M_1}[f,\alpha](\gamma_{x,v}(t),\dot\gamma_{x,v}(t))\,dt.
$$
Following the similar arguments used to derive the estimate (28) in \cite[page 456]{stefanov2004stability}, one can prove the following estimate
$$
\|[w,\phi]\|_{L^2(M_1\setminus M;\Lambda^1_{M_1}\times\C)}\le C\|d_a[w,\phi]\|_{L^2(M_1\setminus M;S^2_{M_1}\times \Lambda^1_{M_1})}\le C \|\mathcal S_a \mathcal E_{M_1}[f,\alpha]\|_{L^2(M_1\setminus M;S^2_{M_1}\times \Lambda^1_{M_1})}.
$$
Then using Korn's inequality in \cite[Corollary~5.12.3]{taylor2011partial},
\begin{align*}
\|[w,\phi]\|_{H^1(M_1\setminus M;\Lambda^1_{M_1}\times\C)}&\le C\big(\|d_a[w,\phi]\|_{L^2(M_1\setminus M;S^2_{M_1}\times \Lambda^1_{M_1})}+\|[w,\phi]\|_{L^2(M_1\setminus M;S^2_{M_1}\times \Lambda^1_{M_1})}\big)\\
&\le C \|\mathcal S_a \mathcal E_{M_1}[f,\alpha]\|_{L^2(M_1\setminus M;S^2_{M_1}\times \Lambda^1_{M_1})}.
\end{align*}
Applying trace theorem and using $[w,\phi]|_{\p M_1}=0$, this implies
$$
\|[w,\phi]\|_{H^{1/2}(\p M;\Lambda^1_M\times\C)}\le C\|\mathcal S_a \mathcal E_{M_1}[f,\alpha]\|_{L^2(M_1\setminus M;S^2_{M_1}\times \Lambda^1_{M_1})}.
$$
Since, in particular, $[w,\phi]\in H^1(M;\Lambda^1_M\times\C)$ solves $\delta_a d_a[w,\phi]=0$ in $M$, by estimate \eqref{stability for solution of elliptic bvp} in Proposition~\ref{solvability of elliptic bvp} we come to
$$
\|[w,\phi]\|_{H^{1}(M;\Lambda^1_M\times\C)}\le C\|\mathcal S_a \mathcal E_{M_1}[f,\alpha]\|_{L^2(M_1\setminus M;S^2_{M_1}\times \Lambda^1_{M_1})}.
$$
Using this together with \eqref{QN=Id+K} and \eqref{solenoidal part of extension of f},
\begin{align*}
\|\cS_a[f,\alpha]\|_{L^2(M;S^2_{M}\times \Lambda^1_{M})}&\le C\|\mathcal S_a \mathcal E_{M_1}[f,\alpha]\|_{L^2(M_1;S^2_{M_1}\times \Lambda^1_{M_1})}\\
&\le C\big(\|Q\tbN_a[f,\alpha]\|_{H^1(M_1;S^2_{M_1}\times \Lambda^1_{M_1})}+\|K[f,\alpha\|_{L^2(M_1;S^2_{M_1}\times \Lambda^1_{M_1})}\big)\\
&\le C\big(\|\tbN_a[f,\alpha]\|_{L^2(M_1;S^2_{M_1}\times \Lambda^1_{M_1})}+\|[f,\alpha\|_{H^{-1}(M_1;S^2_{M_1}\times \Lambda^1_{M_1})}\big),
\end{align*}
where in the last step we used the fact that $K$ is a pseudodifferential operator of order $-1$ in $M_1^{\rm int}$. Since $Q$ is a pseudodifferential operator of order $1$ in $M_1^{\rm int}$, we come to the desired estimate in the statement.
\end{proof}

\begin{proof}[Proof of Theorem~\ref{thm::main 2}]
Part (a) follows from Proposition~\ref{prop::apriori estimate for N} and \cite[Proposition~5.3.1]{taylor1981pseudodifferential}.

To prove part (b), we need the following results. We use the notation
$$
\|(g,a)\|_{C^m(M;S^2_{M}\times \C)}:=\|g\|_{C^m(M; S^2_{M})}+\|a\|_{C^m(M; \C)}.
$$

\begin{Proposition}\label{small perturbation of Delta, S and P}
Given a Riemannian metric $\tilde g$ and $\tilde a\in C^\infty(M;\C)$, there exists sufficiently small $\varepsilon>0$ such that for any metric $g$ and $a\in C^\infty(M;\C)$ with $\|(g,a)-(\tilde g,\tilde a)\|_{C^1(M;S^2_{M}\times \C)}\le\varepsilon$,
$$
\|(-\Delta^D_{g,a})^{-1}-(-\Delta^D_{\tilde g,\tilde a})^{-1}\|\le C\varepsilon,\quad \|\cP_{g,a}-\cP_{\tilde g,\tilde a}\|\le C\varepsilon,\quad\|\cS_{g,a}-\cS_{\tilde g,\tilde a}\|\le C\varepsilon,
$$
with $C>0$ a locally uniform constant depending on $(\tilde g,\tilde a)$ only.
\end{Proposition}
Here and in what follows, $\|\cdot\|$ denotes the operator norms for $L^2(M;S^2_M\times \Lambda^1_M)\to L^2(M;S^2_M\times \Lambda^1_M)$ and $H^{-1}(M;\Lambda^1_M\times \C)\to H^1_0(M;\Lambda^1_M\times \C)$.
\begin{proof}
Suppose that $g$ is a Riemannian metric and $a\in C^\infty(M;\C)$. Then
\begin{equation}\label{difference between inverse of laplacians}
(-\Delta^D_{g,a})^{-1}-(-\Delta^D_{\tilde g,\tilde a})^{-1}=(-\Delta^D_{g,a})^{-1}\big(\Delta_{g,a}-\Delta_{\tilde g,\tilde a}\big)(-\Delta^D_{\tilde g,\tilde a})^{-1}.
\end{equation}
For any $[u,\varphi],[w,\phi]\in H^1_0(M;\Lambda^1_M\times\C)$, we have
\begin{multline*}
|\<(\Delta_{g,a}-\Delta_{\tilde g,\tilde a})[u,\varphi],[w,\phi]\>|\\
=|(d_{\tilde g,\tilde a}[u,\varphi],d_{\tilde g,\tilde a}[w,\phi])_{L^2(M;S^2_M\times\Lambda^1_M)}-(d_{g,a}[u,\varphi],d_{g,a}[w,\phi])_{L^2(M;S^2_M\times\Lambda^1_M)}|\\
\le C\|(g,a)-(\tilde g,\tilde a)\|_{C^1(M;S^2_{M}\times \C)}\big(\|(g,a)\|_{C^1(M;S^2_{M}\times \C)}+\|(\tilde g,\tilde a)\|_{C^1(M;S^2_{M}\times \C)}\big)\\
\times\|[u,\varphi]\|_{H^1(M;\Lambda^1\times\C)}\|[w,\phi]\|_{H^1(M;\Lambda^1_M\times\C)}.
\end{multline*}
Hence
$$
\|\Delta_{g,a}-\Delta_{\tilde g,\tilde a}\|\le C\|(g,a)-(\tilde g,\tilde a)\|_{C^1(M;S^2_{M}\times \C)}\big(\|(g,a)\|_{C^1(M;S^2_{M}\times \C)}+\|(\tilde g,\tilde a)\|_{C^1(M;S^2_{M}\times \C)}\big).
$$
Suppose $\|(g,a)-(\tilde g,\tilde a)\|_{\cG^1(M)}\le\varepsilon$ for sufficiently small $\varepsilon>0$. Then
$$
\|(g,a)\|_{C^1(M;S^2_{M}\times \C)}\le\varepsilon+\|(\tilde g,\tilde a)\|_{C^1(M;S^2_{M}\times \C)}\quad\text{and}\quad \|(-\Delta^D_{g,a})^{-1}\|\le\|(-\Delta^D_{\tilde g,\tilde a})^{-1}\|\big(1+C\varepsilon\|(-\Delta^D_{\tilde g,\tilde a})^{-1}\|\big),
$$
where $C>0$ is a uniform constant in an $\varepsilon$-ball of $(\tilde g,\tilde a)$ in $\cG^1$-norm. Hence
$$
\|(-\Delta^D_{g,a})^{-1}\|\le C_1(1-C\varepsilon)^{-1}\quad\text{with}\quad C_1:=\|(-\Delta^D_{\tilde g,\tilde a})^{-1}\|.
$$
This together with \eqref{difference between inverse of laplacians} and, implies that
$$
\|(-\Delta^D_{g,a})^{-1}-(-\Delta^D_{\tilde g,\tilde a})^{-1}\|\le C\|(g,a)-(\tilde g,\tilde a)\|_{C^1(M;S^2_{M}\times \C)}\le C\varepsilon
$$
as desired. This then can be used to prove the corresponding estimates for $\cP_{g,a}$ and $\cS_{g,a}$.
\end{proof}

\begin{Proposition}\label{N and N' are close}
Let $(\widetilde M,g)$ be a simple manifold and let $a\in C^\infty(\widetilde M;\C)$. Suppose that a metric $\tilde g$ and $\tilde a\in C^\infty(\widetilde M;\C)$ satisfy $\|(g,a)-(\tilde g,\tilde a)\|_{C^3(\widetilde M;S^2_{\widetilde M}\times \C)}\le \varepsilon$ for sufficiently small $\varepsilon>0$. Then $(\widetilde M,\tilde g)$ is simple and the following estimate holds
$$
\|(\widetilde \bN_{g,a}-\widetilde \bN_{\tilde g,\tilde a})[f,\alpha]\|_{H^1(\widetilde M;S^2_{\widetilde M}\times \Lambda^1_{\widetilde M})}\le C\varepsilon\|[f,\alpha]\|_{L^2(M;S^2_M\times \Lambda^1_M)},\qquad [f,\alpha]\in L^2(M;S^2_M\times \Lambda^1_M)
$$
for some $C>0$ constant depending only on $(g,a)$.
\end{Proposition}
\begin{proof}
This can be proven following similar arguments as in \cite[Proposition~5.1]{frigyik2008x} and \cite[Proposition~3]{holman2010doppler}; see also \cite[Theorem~8]{holman2013generic}. One needs to show is that the generators of the geodesic flows related to $g$ and $\tilde g$ are $C\varepsilon$ close in $C^2$. Also, one needs $\|\widetilde U_a-\widetilde U_{\tilde a}\|_{C^2(S\widetilde M)}\le C\varepsilon$. These follow from our assumption $\|(g,a)-(\tilde g,\tilde a)\|_{C^3(\widetilde M;S^2_{\widetilde M}\times \C)}\le \varepsilon$.
\end{proof}

For a given $[f,\alpha]\in L^2(M;S^2_M\times \Lambda^1_M)$, we use part (a), Proposition~\ref{small perturbation of Delta, S and P} and Proposition~\ref{N and N' are close},
\begin{multline*}
\|\cS_{\tilde g,\tilde a}[f,\alpha]\|_{L^2(M;S^2_M\times \Lambda^1_M)}\le \|\cS_{g,a}\cS_{\tilde g,\tilde a}[f,\alpha]\|_{L^2(M;S^2_M\times \Lambda^1_M)}+\|(\cS_{\tilde g,\tilde a}-\cS_{g,a})\cS_{\tilde g,\tilde a}[f,\alpha]\|_{L^2(M;S^2_M\times \Lambda^1_M)}\\
\le C\|\widetilde\bN_{g,a}\cS_{\tilde g,\tilde a}[f,\alpha]\|_{H^1(\widetilde M;S^2_{\widetilde M}\times \Lambda^1_{\widetilde M})}+C\varepsilon\|\cS_{\tilde g,\tilde a}[f,\alpha]\|_{L^2(M;S^2_M\times \Lambda^1_M)}\\
\le C\|\widetilde\bN_{\tilde g,\tilde a}\cS_{\tilde g,\tilde a}[f,\alpha]\|_{H^1(\widetilde M;S^2_{\widetilde M}\times \Lambda^1_{\widetilde M})}+C\varepsilon\|\cS_{\tilde g,\tilde a}[f,\alpha]\|_{L^2(M;S^2_M\times \Lambda^1_M)}\\
\quad+C\|(\widetilde\bN_{g,a}-\widetilde\bN_{\tilde g,\tilde a})\cS_{\tilde g,\tilde a}[f,\alpha]\|_{H^1(\widetilde M;S^2_{\widetilde M}\times \Lambda^1_{\widetilde M})}\\
\le C\|\widetilde\bN_{\tilde g,\tilde a}\cS_{\tilde g,\tilde a}[f,\alpha]\|_{H^1(\widetilde M;S^2_{\widetilde M}\times \Lambda^1_{\widetilde M})}+C\varepsilon\|\cS_{\tilde g,\tilde a}[f,\alpha]\|_{L^2(M;S^2_M\times \Lambda^1_M)}.
\end{multline*}
Since $\widetilde\bN_{\tilde g,\tilde a}\cS_{\tilde g,\tilde a}[f,\alpha]=\widetilde\bN_{\tilde g,\tilde a}[f,\alpha]$, fixing sufficiently small $\varepsilon>0$, we get
$$
\|\cS_{\tilde g,\tilde a}[f,\alpha]\|_{L^2(M;S^2_M\times \Lambda^1_M)}\le C\|\widetilde\bN_{\tilde g,\tilde a}[f,\alpha]\|_{H^1(\widetilde M;S^2_{\widetilde M}\times \Lambda^1_{\widetilde M})},
$$
where $C>0$ depends only on $(g,a)$.
\end{proof}

\section{Generic $s$-injectivity}\label{sctn::injectivity result}
The present section contains the proof of Theorem~\ref{thm::main 1}. The notation $\WF_A([f,\alpha])$ stands for the analytic wave front set of the pair $[f,\alpha]$; see \cite{sjostrand1982singularites,treves1980introduction}.

\begin{Proposition}\label{prop::analytic microlocal injectivity}
Suppose that a simple manifold $(M,g)$ and $a:M\to \C$ are real analytic. For a given $(x_0,\xi_0)\in T^*M^{\rm int}\setminus\{0\}$ let $\gamma_0$ be a geodesic through $x_0$ and normal to $\xi_0$. If $[f,\alpha]\in L^2(M;S^2_M\times \Lambda^1_M)$ satisfies $\bI_{a}[f,\alpha](\gamma)=0$ for all $\gamma$ near $\gamma_0$ and $\delta_a[f,\alpha]=0$ near $x_0$, then $(x_0,\xi_0)\notin \WF_A([f,\alpha])$.
\end{Proposition}
\begin{proof}
Without loss of generality, we assume that $\gamma_0:[\ell^-,\ell^+]\to \widetilde M$ with $\ell^-<0<\ell^+$, $x_0=\gamma_0(0)$ and $\gamma_0(\ell^-),\gamma_0(\ell^+)\in \widetilde M^{\rm int}\setminus M$. As it was explained in \cite[Section~2.1]{stefanov2008nonsimple}, we can work in a tubular neighborhood $U$ of $\gamma_0$ in $\widetilde M$ with analytic coordinates $x=(x',t)$, with $x'=(x^1,\dots,x^{n-1})$, so that $U=\{(x',t):|x'|<\varepsilon,\,\ell^--\varepsilon< t<\ell^++\varepsilon\}$ for some $0<\varepsilon\ll 1$, $x_0=0$ and $\gamma_0([\ell^-,\ell^+])=\{(0,\dots,0,t):t\in[\ell^-,\ell^+]\}$. If $\varepsilon>0$ is sufficiently small, we have $(x',\ell^-),(x',\ell^+)\in\widetilde M\setminus M$ for $|x'|<\varepsilon$. We also can assume that $g_{ij}(0)=\delta_{ij}$ and $\xi_0=(\xi'_0,0)$. Then $v_0:=\dot\gamma_0(0)=(0,\dots,0,1)$ and hence $\gamma_0=\gamma_{x_0,v_0}$.

We parameterize curves near $\gamma_0$ using the above mentioned analytic coordinates. For $|x'|<2\varepsilon/3$ and $|\theta'|\ll 1$, we write
$$
\gamma_{x',\theta'}(t):=\gamma_{(x',0),\omega(\theta')}(t),\qquad\omega(\theta'):=(\theta',1)/|(\theta',1)|,
$$
which will stay in $U$ for all $t\in [\ell^-,\ell^+]$ and $\gamma_{x',\theta'}(\ell^-),\gamma_{x',\theta'}(\ell^+)\in\widetilde M\setminus M$.

Following \cite{stefanov2008nonsimple}, we work with a sequence of cut-off functions $\chi_N\in C^\infty_0(\R^{n-1})$, $N\ge 1$ integer, such that $\chi_N(x')\equiv 1$ for $|x'|\le \varepsilon/3$, $\supp(\chi_N)\subset\{x'\in \R^{n-1}:|x'|<2\varepsilon/3\}$ and
\begin{equation}\label{ineq::estimates for derivatives of chi_N}
|\p^\alpha\chi_N(x')|\le (CN)^{|\alpha|}\quad\text{for all}\quad x'\in\R^{n-1}\quad\text{and}\quad|\alpha|<N,
\end{equation}
for some constant $C>0$ independent of $N$; see \cite[Lemma~1.1]{treves1980introduction} for the existence of such cut-off functions.

Let $\lambda>0$ be a large parameter and $\xi=(\xi',\xi^n)$ be in a sufficiently small complex neighborhood of $\xi_0$. Then for $|\theta'|\ll 1$, multiplying $\bI_a[f,\alpha](\gamma_{x',\theta'})=0$ by $e^{i\lambda x'\cdot\xi'}\chi_N(x')$ and integrating with respect to $x'$, we obtain
$$
\int e^{i\lambda x'\cdot\,\xi'}\chi_N(x')\int \tU_{-a}(\gamma_{x',\theta'},\dot \gamma_{x',\theta'})[f,\alpha](\gamma_{x',\theta'},\dot \gamma_{x',\theta'})\,dt\,dx'=0.
$$
Since $\supp(\chi_N)\subset\{x'\in \R^{n-1}:|x'|<2\varepsilon/3\}$, we can assume that local coordinates near $\gamma_0$ are given by $x=\gamma_{x',\theta'}(t)$ for fixed $|\theta'|\ll 1$. If $\theta'=0$ we clearly have $x=(x',t)$. By perturbation arguments, one can see that $(x',t)$ are analytic local coordinates which depend analytically on $\theta'$. As a result, we have $x=(x'+t\theta',t)+O(|\theta'|)$. In a sufficiently small complex neighborhood of $\xi_0$, we write $\theta'=\theta'(\xi)$ analytically depending on $\xi$ and such that $\theta'(\xi)\cdot\xi'=0$ and $\theta'(\xi_0)=0$.

Using these change of variables,
\begin{equation}\label{spacial integral of ray transform after change of variables}
\int e^{i\lambda \varphi(x,\xi)}u_N(x,\xi) [f,\alpha](x,b(x,\xi))\,dx=0,
\end{equation}
where $\varphi$ is the phase function given by
\begin{equation}\label{def of phase phi}
\varphi(x,\xi):=x'(x,\theta'(\xi))\cdot\,\xi'.
\end{equation}
The function $u_N$ and the vector field $b$ are both analytic for $x$ and $\xi$ near $\gamma_0$ and $\xi_0$, respectively. Moreover, $u_N$ vanishes outside $U$ and satisfies \eqref{ineq::estimates for derivatives of chi_N}. Also, $b(0,\xi)=\omega(\theta'(x,\xi))$ and $u_N(0,\xi)=\widetilde U_{-a}(0,\omega(\theta'(x,\xi)))$.

We need the following result which was proven in \cite{stefanov2008nonsimple}.
\begin{Lemma}\label{local uniqueness of critical point}
For the phase function $\varphi$, given by \eqref{def of phase phi}, there is $\delta>0$ such that if $\p_\xi\varphi(x,\xi)=\p_\xi\varphi(y,\xi)$ for some $x\in U$, $|y|<\delta$ and $|\xi-\xi_0|<\delta$, then $x=y$.
\end{Lemma}
Suppose $|y|<\delta$ and $|\eta-\xi_0|<\delta/2$. Consider $\rho\in C^\infty_0(\R^n)$ such that $\supp(\rho)\subset\{\xi\in\R^n:|\xi|<\delta\}$ and $\rho(\xi)=1$ for $|\xi|<\delta/2$. Multiplying \eqref{spacial integral of ray transform after change of variables} by
$$
\rho(\xi-\eta)e^{i\lambda\big(\frac i2(\xi-\eta)^2-\varphi(y,\xi)\big)}
$$
and integrating with respect to $\xi$, we obtain
\begin{equation}\label{spacial and directional integral of ray transform after change of variables}
\int\int e^{i\lambda \Phi(x,y,\xi,\eta)}U_N(x,\xi,\eta)[f,\alpha](x,b(x,\xi))\,dx\,d\xi=0,
\end{equation}
where
$$
\Phi(x,y,\xi,\eta):=\frac i2(\xi-\eta)^2+\varphi(x,\xi)-\varphi(y,\xi),\quad U_N(x,\xi,\eta):=\rho(\xi-\eta)u_N(x,\xi).
$$
According to Lemma~\ref{local uniqueness of critical point}, there is a constant $C_0>0$ such that the function $\xi\mapsto \Phi(x,y,\xi,\eta)$ has no critical points when $|x-y|>C_0\delta$. Therefore, we can estimate
\begin{equation}\label{estimate for integral of ray transform or x NOT close to y}
\int\int_{|x-y|>C_0\delta} e^{i\lambda \Phi(x,y,\xi,\eta)}U_N(x,\xi,\eta)[f,\alpha](x,b(x,\xi))\,dx\,d\xi=O((C'N/\lambda)^N+Ne^{-\lambda/C'})
\end{equation}
for some constant $C'>0$. To get the estimate \eqref{estimate for integral of ray transform or x NOT close to y}, we integrate by parts $N$ times with respect to $\xi$ using the identity
$$
e^{i\lambda \Phi(x,y,\xi,\eta)}=\frac{1}{i\lambda|\p_\xi\Phi|^2}\p_\xi \overline\Phi\cdot \p_\xi e^{i\lambda \Phi(x,y,\xi,\eta)}
$$
together with boundedness of $|\p_\xi\Phi|$ from below on the region of integration. We also used the facts that on the boundary of the region of integration, the function $e^{i\lambda \Phi(x,y,\xi,\eta)}$ is exponentially small in $\lambda$, and $U_N(x,\xi,\eta)$ satisfies an estimate like \eqref{ineq::estimates for derivatives of chi_N} in $\xi$.

To estimate the integral in \eqref{spacial and directional integral of ray transform after change of variables} for $|x-y|\le C_0\delta$, we study the critical points of $\xi\mapsto \Phi(x,y,\xi,\zeta)$. One can see that $\p_\xi \Phi(x,y,\xi,\eta)=i(\xi-\eta)+\p_\xi\varphi(x,\xi)-\p_\xi\varphi(y,\xi)$. If $x=y$, the function $\Phi$ has the unique critical point $\xi_c=\eta$ which is non-degenerate. By Lemma~\ref{local uniqueness of critical point}, the function $\Phi$ has at most one critical point $\xi_c=\xi_c(x,y,\eta)$, depending analytically on $x$, $y$ and $\eta$, if $|x-y|\le C_0\delta$. Furthermore, $\Im\big(\p_\xi^2\Phi(x,x,\eta,\xi)\big)=\id>0$ and hence $\Im\big(\p_\xi^2\Phi(x,y,\eta,\xi_c)\big)>0$ when $|x-y|\le C_0\delta$ for sufficiently small $C_0>0$. Taking $C_0>0$ possibly smaller, we can ensure that $U_N$ is still analytic and independent of $N$ within $|x-y|\le C_0\delta$. Therefore, using complex stationary phase method \cite[Theorem~2.8]{sjostrand1982singularites} (and Remark~2.10 in there) to \eqref{spacial and directional integral of ray transform after change of variables}, we get
\begin{equation}\label{estimate for spacial integral of ray transform for x close to y}
\begin{aligned}
\int\int_{|x-y|\le C_0\delta} e^{i\lambda \Phi(x,y,\xi,\eta)}U&_N(x,\xi,\eta)[f,\alpha](x,b(x,\xi))\,dx\,d\xi\\
&=\int_{|x-y|\le C_0\delta} e^{i\lambda \Psi(x,y,\eta)} U_N(x,\xi,\eta)[f,\alpha](x,b(x,\xi))\,dx+O(e^{-C''\lambda})
\end{aligned}
\end{equation}
for all $N>0$ and for some $C''>0$. Define
$$
\Psi(x,y,\eta):=\Phi(x,y,\xi_c,\eta).
$$
Then $\Psi(x,x,\eta)=0$, $\p_x\Psi(x,x,\eta)=\p_x\varphi(x,\eta)$, $\p_y\Psi(x,x,\eta)=-\p_x\varphi(x,\eta)$ and $\Im(\Psi(x,y,\eta))>|x-y|^2/C$. Combining \eqref{spacial and directional integral of ray transform after change of variables}--\eqref{estimate for spacial integral of ray transform for x close to y}, we get
\begin{equation}\label{estimate for spacial integral of ray transform for x close to y after xi_c}
\int_{|x-y|\le C_0\delta} e^{i\lambda \Psi(x,y,\eta)} u_N(x,y,\eta;\lambda) [f,\alpha](x,B(x,y,\eta))\,dx=O((C'N/\lambda)^N+Ne^{-C\lambda})
\end{equation}
for all $N>0$, where the function $u_N$ and the vector field $B$ are analytic. Note that the left side of \eqref{estimate for spacial integral of ray transform for x close to y} is in fact independent of $N$ on $|x-y|\le C_0\delta$. Then choosing $N$ such that $N\le\lambda/(C'e)\le N+1$, we get that the right side of \eqref{estimate for spacial integral of ray transform for x close to y after xi_c} is $O(e^{-\lambda/C})$.

It was shown in \cite{stefanov2008nonsimple} that $\p_\xi\p_y\varphi(0,\xi_0)=\id$ and hence $\varphi$ is a non-degenerate near $(0,\xi_0)$. Therefore, we can make a change of variables $(y,\eta)\mapsto \beta=(y,\zeta)$, with $\zeta:=\p_y\varphi(y,\eta)$, in a small enough neighborhood of $(0,\xi_0)$. Then plugging $\eta=\eta(\beta)$ in \eqref{estimate for spacial integral of ray transform for x close to y after xi_c}, we get
\begin{equation}\label{estimate for spacial integral of ray transform for x close to y version 2}
\int_{|x-y|\le C_0\delta} e^{i\lambda \tilde\Psi(x,\beta)}\tilde u(x,\beta;\lambda)[f,\alpha](x,\tilde B(x,\beta))\,dx=O(e^{-\lambda/C}),
\end{equation}
where $\tilde\Psi$, $\tilde u$ and $\tilde B$ are analytic and have the same properties as $\Psi$, $u_N$ and $B$. In particular,
$$
\tilde\Psi(x,x,\zeta)=0,\quad \p_x\tilde\Psi(x,x,\zeta)=\zeta,\quad \p_y\tilde\Psi(x,x,\zeta)=-\zeta.
$$
Define
$$
P^{ij}(x,\beta;\lambda):=\tilde u(x,\beta;\lambda)\tilde B^i(x,\beta)\tilde B^j(x,\beta)\quad Q^i(x,\beta;\lambda):=\tilde u(x,\beta;\lambda)\tilde B^i(x,\beta).
$$
Then \eqref{estimate for spacial integral of ray transform for x close to y version 2} can be rewritten as
\begin{equation}\label{estimate for spacial integral of ray transform for x close to y version 3}
\int_{|x-y|\le C_0\delta} e^{i\lambda \tilde\Psi(x,\beta)}\big(P^{ij}(x,\beta;\lambda)f_{ij}(x)+Q^i(x,\beta;\lambda)\alpha_i(x)\big)\,dx=O(e^{-\lambda/C}).
\end{equation}
Note that $\tilde B(0,0,\xi_0)=v_0=(0,\dots,0,1)$ and
\begin{align*}
\sigma_p(P^{ij})(0,0,\xi_0)&=\widetilde U_{-a}\big(0,\tilde B(0,0,\xi_0)\big)\tilde B^i(0,0,\xi_0)\tilde B^j(0,0,\xi_0)=\widetilde U_{-a}(0,v_0)\,v_0^i v_0^j,\\
\sigma_p(Q^i)(0,0,\xi_0)&=\widetilde U_{-a}\big(0,\tilde B(0,0,\xi_0)\big)\tilde B^i(0,0,\xi_0)=\widetilde U_{-a}(0,v_0)\,v_0^i.
\end{align*}

Let $v_0,v_1,\dots,v_{N-1}$ be $N=(n-1)+n(n-1)/2$ unit vectors at $x_0=0$ such that $v_k\perp\xi_0$, $k=0,\dots,N-1$, and any symmetric $2$-tensor $f$ and $1$-tensor $\alpha$ with $f_{ij}(\xi_0)^j=0$ for all $i=1,\dots,n$ and $\alpha_j(\xi_0)^j=0$, can be uniquely determined by $f_{ij}v^iv^j+\alpha_i v^i$, $v=v_0,\dots,v_{N-1}$; see \cite[Lemma~3.3]{dairbekov2007boundary}. Such vectors exist in any open set in $\xi_0^\perp$; see~\cite{stefanov2008nonsimple}. Hence, we can assume that $v_k$ is in a sufficiently small neighborhood of $v_0$. Moreover, for each geodesic $\gamma_{x_0,v_k}$, $k=0,\dots,N-1$, $\gamma_{x_0,v_k}([\ell^-,\ell^+])\subset U$ and $\gamma_{x_0,v_k}(\ell^-),\gamma_{x_0,v_k}(\ell^+)\in\widetilde M^{\rm int}\setminus M$. Then, after rotating the coordinate system so that $v_k=(0,\dots,0,1)$, we repeat the above construction and get $N$ phase functions $\tilde\Psi_k$ and symbols $P_k$, $k=0,\dots,N-1$, such that
\begin{equation}\label{estimate for spacial integral of ray transform for x close to y N times}
\int_{|x-y|\le C_0\delta} e^{i\lambda \tilde\Psi_k(x,\beta)}\big(P^{ij}_k(x,\beta;\lambda)f_{ij}(x)+Q^i_k(x,\beta;\lambda)\alpha_i(x)\big)\,dx=O(e^{-\lambda/C}) 
\end{equation}
for $k=0,\dots,N-1$, where $\Psi_0$, $P^{ij}_0$ and $Q^i_0$ are exactly those appearing in \eqref{estimate for spacial integral of ray transform for x close to y version 3}. Note also that
$$
\sigma_p(P^{ij}_k)(0,0,\xi_0)=\widetilde U_{-a}(0,v_k)\,v_k^i v_k^j,\quad \sigma_p(Q^i_k)(0,0,\xi_0)=\widetilde U_{-a}(0,v_k)\,v_k^i,\qquad k=0,\dots,N-1.
$$
We need $n+1$ more equations in order to turn \eqref{estimate for spacial integral of ray transform for x close to y N times} into an elliptic system of $n+1+N=n+n(n+1)/2$ equations. For this, recall that $\delta_a[f,\alpha](x)=0$ near $x_0=0$. Following \cite{stefanov2008nonsimple}, consider $\chi_0\in C^\infty_0(M^{\rm int})$ with $\chi_0\equiv 1$ near $x_0=0$. Then integrating $\frac{1}{\lambda}e^{i\lambda\tilde\Psi_0(x,\beta)}\chi_0(x)\delta_a[f,\alpha](x)=0$ with respect to $x$, applying the integration by parts and using $\Im(\tilde\Psi_0(x,\beta))>|x-y|^2/C$, we get
\begin{equation}\label{f is divergence free}
\int e^{i\lambda\tilde\Psi_0(x,\beta)} \big(R^{j}(x,\beta;\lambda)f_{ij}(x)+V(x,\beta;\lambda)\alpha_i(x)\big)\,dx=0,\quad i=1,\dots,n,
\end{equation}
and
\begin{equation}\label{alpha is divergence free}
\int e^{i\lambda\tilde\Psi_0(x,\beta)} \big(W^{ij}(x,\beta;\lambda)f_{ij}(x)+S^j(x,\beta;\lambda)\alpha_j(x)\big)\,dx=0,
\end{equation}
where $\sigma_p(R^{j})(0,0,\xi_0)=\sigma_p(S^{j})(0,0,\xi_0)=(\xi_0)^j$, $\sigma_p(V)(0,0,\xi_0)=0$ and $\sigma_p(W^{ij})(0,0,\xi_0)=0$.

Let us now write the system, consisting of \eqref{estimate for spacial integral of ray transform for x close to y N times}, \eqref{f is divergence free} and \eqref{alpha is divergence free}, as
\begin{equation}\label{A system}
\int_{|x-y|\le C}\diag(e^{i\lambda \tilde\Psi_0},\dots,e^{i\lambda \tilde\Psi_{N-1}},\underbrace{e^{i\lambda \tilde\Psi_0},\dots,e^{i\lambda \tilde\Psi_0}}_{(1+n)\text{-times}})(x,\beta)\mathbf A(x,\beta;\lambda)[f,\alpha](x)\,dx=O(e^{-\lambda/C}),
\end{equation}
where $\mathbf A(x,y,\zeta;\lambda)$ is a matrix valued symbol acting on $[f,\alpha]$.

Now, we prove that \eqref{A system} is elliptic at $(0,0,\xi_0)$. For this, suppose that $f$ is a symmetric $2$-tensor and $\alpha$ is a $1$-form such that $\sigma_p(\mathbf A)(0,0,\xi_0)=0$. Then by looking at principal symbols in \eqref{estimate for spacial integral of ray transform for x close to y N times}, \eqref{f is divergence free} and \eqref{alpha is divergence free}, one can see that this is equivalent to
$$
f_{ij}v_k^i v_k^j+\alpha_i v_k^i=0,\,\, k=0,\dots,N-1,\quad\text{and}\quad f_{ij}(\xi_0)^j=0,\,\, i=1,\dots,n,\quad \alpha_j(\xi_0)^j=0.
$$
Following the ideas as in the end of the proof of Proposition~\ref{prop::ellipticity of N}, we can show that this implies $f=0$ and $\alpha=0$.

Finally, we need to replace $\mathbf A$ in \eqref{A system} by the identity matrix $\bid$ and all phase functions by the same phase $\tilde\Phi_0$. Then this would show that $(0,\xi_0)=(x_0,\xi_0)\notin\WF_A([f,\alpha])$ in the sense of \cite[Definition~6.1]{sjostrand1982singularites}. For this, we need to modify the proof of \cite[Proposition~6.2]{sjostrand1982singularites} to the case of matrix-valued symbols following \cite{stefanov2008nonsimple}. Consider the operator $\Op(\mathbf A)$ given by
$$
\Op(\mathbf A)[f,\alpha](y)=\int\int \diag\big(e^{i\lambda(\tilde\Phi_j(y,\beta)-\overline{\tilde\Phi_j(x,\beta)})}\big) \mathbf A(x,\beta;\lambda)[f,\alpha](x)\,dx\,d\beta,
$$
where $\tilde \Phi_j=\tilde \Psi_j$ for $j=0,\dots,N-1$ and $\tilde \Phi_j=\tilde \Psi_0$ for $j=N,\dots,N+n$. This is a pseudodifferential operator with an elliptic principal symbol. Therefore, there is an analytic classical matrix-valued symbol $\mathbf R(x,\beta;\lambda)$, defined near $(0,0,\xi_0)$, such that
$$
\Op(\mathbf A)\big(\mathbf R(\cdot,\beta;\lambda) e^{i\lambda \tilde\Phi_0}\big)(y)=\bid e^{i\lambda \tilde\Phi_0(y,\beta)}
$$
for $\beta$ in a neighborhood of $(0,\xi_0)$.  Following the same argument as is in the proof of \cite[Proposition~6.2]{sjostrand1982singularites}, we can show that $\bid e^{i\lambda \tilde\Phi_0}$ can be expressed as a superposition of $\mathbf A e^{i\lambda \tilde\Phi_0}$ modulo an exponentially decreasing function. Then the rest of the proof is identical to that of \cite[Proposition~6.2]{sjostrand1982singularites} which, with a possible new constant $C>0$, gives
$$
\int e^{i\lambda\tilde\Phi_0(x,\beta)}\chi(x)\bid [f,\alpha](x)\,dx=O(e^{-\lambda/C})
$$
for $\beta$ in a neighborhood of $(0,\xi_0)$ and for some cut-off function near $x_0=0$. This proves our claim that $(x_0,\xi_0)\notin \WF_A([f,\alpha])$.
\end{proof}

By $\cA(M;S^2_M\times \Lambda^1_M)$ and $\cA(M;\Lambda^1_M\times \C)$ we denote the space of real analytic pairs on $M$. Analogous notations are used on $\widetilde M$. We need the following two lemmas for the proof of Theorem~\ref{thm::main 1}.

\begin{Lemma}\label{lemma::If=0 implies f is analytic}
Assume that $(M,g)$ is a real analytic simple manifold and $a:M\to\C$ is real analytic. If $\bI_a [f,\alpha]=0$ with $[f,\alpha]\in L^2(M;S^2_M\times \Lambda^1_M)$, then $\cS_a[f,\alpha]\in\cA(M;S^2_M\times \Lambda^1_M)$.
\end{Lemma}
\begin{proof}
According to Proposition~\ref{prop::analytic microlocal injectivity}, $\cS_a[f,\alpha]$ is analytic in $M^{\rm int}$. Our aim is to show that it is analytic up to $\p M$. For this, consider extension $\widetilde M$ of $M$ as in Section~\ref{section::normal operator} which can be chosen to be real analytic. We also extend $g$ and $a$ to $\widetilde M$ to be real analytic and so that $(\widetilde M,g)$ is simple.

First, we show that $\cS_a\cE_{\widetilde M}[f,\alpha]\in \cA(\widetilde M;S^2_{\widetilde M}\times \Lambda^1_{\widetilde M})$. According to the assumption $\bI_a[f,\alpha]=0$, we have $\tbI_a \cS_a\cE_{\widetilde M}[f,\alpha]=0$. Applying Proposition~\ref{prop::analytic microlocal injectivity} to $\widetilde M$, $\cS_a\cE_{\widetilde M}[f,\alpha]$ is real analytic in $\widetilde M^{\rm int}$. Thus, we need to show that $\cS_a\cE_{\widetilde M}[f,\alpha]$ is real analytic up to $\p \widetilde M$. Observe that $\cS_a\cE_{\widetilde M}[f,\alpha]=-d_a[w_{\widetilde M},\phi_{\widetilde M}]$ in $\widetilde M\setminus M$, where $[w_{\widetilde M},\phi_{\widetilde M}]:=\cP_a\cE_{\widetilde M}[f,\alpha]\in H^1(\widetilde M;\Lambda^1_{\widetilde M}\times\C)$. In particular, $[w_{\widetilde M},\phi_{\widetilde M}]$ satisfies $(-\Delta_{g,a})[w_{\widetilde M},\phi_{\widetilde M}]=0$ in $\widetilde M\setminus M$ with $[w_{\widetilde M},\phi_{\widetilde M}]|_{\p \widetilde M}=0$. Then $[w_{\widetilde M},\phi_{\widetilde M}]$ is real analytic up to $\p \widetilde M$; this can be done as in \cite[Lemma~3]{stefanov2005boundary} using results of \cite{morrey1957analyticity}. This gives $\cS_a\cE_{\widetilde M}[f,\alpha]\in \cA(\widetilde M;S^2_{\widetilde M}\times \Lambda^1_{\widetilde M})$.

Next, we compare $\cS_a\cE_{\widetilde M}[f,\alpha]$ and $\cS_a[f,\alpha]$. Since $d_a[w_{\widetilde M},\phi_{\widetilde M}]=-\cS_a\cE_{\widetilde M}[f,\alpha]$ in $\widetilde M\setminus M$ and $\cS_a\cE_{\widetilde M}[f,\alpha]\in \cA(\widetilde M;S^2_{\widetilde M}\times \Lambda^1_{\widetilde M})$, the pair $d_a[w_{\widetilde M},\phi_{\widetilde M}]$ is real analytic in $\widetilde M\setminus M$ (up to $\p \widetilde M$). Integrating
$$
\frac{d}{dt}\Big(\widetilde U^{-1}_a(\gamma(t),\dot\gamma(t))[w_{\widetilde M},\phi_{\widetilde M}](\gamma(t),\dot\gamma(t))\Big)=\widetilde U^{-1}_a(\gamma(t),\dot\gamma(t))d_a[w_{\widetilde M},\phi_{\widetilde M}](\gamma(t),\dot\gamma(t))
$$
along geodesics $\gamma$ in $\widetilde M\setminus M$ connecting points on $\p M$ and $\p \widetilde M$, we get real analyticity of $[w_{\widetilde M},\phi_{\widetilde M}]|_{\p M}$ on $\p M$.

Since $\cS_a[f,\alpha]=[f,\alpha]-d_a[w,\phi]$ in $M$ and $\cS_a\cE_{\widetilde M}[f,\alpha]=\cE_{\widetilde M}[f,\alpha]-d_a[w_{\widetilde M},\phi_{\widetilde M}]$ in $\widetilde M$, we can write
\begin{equation}\label{eqn::expression for solenoidal part of zero extension}
\cS_a[f,\alpha]=\cS_a\cE_{\widetilde M}[f,\alpha]+d_a[v,\varphi]\quad\text{in}\quad M
\end{equation}
where $[w,\phi]:=\cP_a [f,\alpha]$ and $[v,\varphi]:=[w_{\widetilde M},\phi_{\widetilde M}]-[w,\phi]$. Then $[v,\varphi]$ is a solution for
$$
-\Delta_{g,a}[v,\varphi]=0\quad\text{in}\quad M,\qquad [v,\varphi]|_{\p M}=[w_{\widetilde M},\phi_{\widetilde M}]|_{\p M}.
$$
Since $[w_{\widetilde M},\phi_{\widetilde M}]|_{\p M}$ is real analytic on $\p M$, we have $[v,\varphi]\in\cA(M;\Lambda^1_M\times\C)$. By \eqref{eqn::expression for solenoidal part of zero extension}, this implies $\cS_a[f,\alpha]\in\cA(M;S^2_M\times \Lambda^1_M)$.
\end{proof}

\begin{Lemma}\label{lemma::modifying f to have zero jet}
Suppose that $a\in C^\infty(M;\C)$. If $\bI_a [f,\alpha]=0$ with $[f,\alpha]\in C^\infty(M;S^2_M\times \Lambda^1_M)$, then there is $[w,\phi]\in C^\infty(M;\Lambda^1_M\times\C)$, with $[w,\phi]|_{\p M}=0$, such that for $[\tilde f,\tilde \alpha]:=[f,\alpha]-d_a[w,\phi]$ we have
\begin{equation}\label{eqn::modifying f to have zero jet}
\p^m [\tilde f,\tilde \alpha]|_{\p M}=0
\end{equation}
for all multi-indices $m$, and in boundary normal coordinates,
\begin{equation}\label{eqn::tilde f_in=tilde alpha_n=0}
\tilde f_{i n}=\tilde \alpha_n=0,\quad i=1,\dots,n.
\end{equation}
\end{Lemma}
\begin{proof}
We start with construction of $[\tilde f,\tilde \alpha]$ satisfying \eqref{eqn::tilde f_in=tilde alpha_n=0}. Let $(x',x^n)$ be a boundary normal coordinate near a boundary point, i.e. $x^n>0$ in $M^{\rm int}$ and $x^n=0$ defines $\p M$. In these coordinates, we have
$$
g_{in}=\delta_{in},\quad \Gamma_{in}^n=\Gamma^i_{nn}=0,\quad i=1,\dots,n.
$$
Then \eqref{eqn::tilde f_in=tilde alpha_n=0} is equivalent to
\begin{equation}\label{eqn::system of ODEs}
\frac12(\p_i w_n+\p_n w_i-2\Gamma_{in}^j w_j)+a\phi\delta_{in}=f_{in},\quad \p_n\phi+aw_n=\alpha_n,\quad i=1,\dots,n.
\end{equation}
We solve the first system by setting $i=n$ and solving the system of ODEs
$$
\p_n w_n+a\phi=f_{nn},\quad \p_n\phi+aw_n=\alpha_n
$$
with the initial conditions $w_n|_{x^n=0}=\phi|_{x^n=0}=0$. Then we solve the remaining system of $(n-1)$-ODEs with initial conditions
$$
\p_n w_\imath-2\Gamma_{\imath n}^\kappa w_\kappa=2f_{\imath n}-\p_\imath w_n\quad\text{and}\quad w_\imath|_{x^n=0}=0,\quad \imath=1,\dots,n,
$$
where $\kappa$ runs from $1$ to $n-1$. This gives the construction of $[w,\phi]$ near $\p M$. Multiplying $[w,\phi]$ by a proper cut-off function we can assume that $[w,\phi]$ is globally defined on $M$. Then we define $[\tilde f,\tilde \alpha]:=[f,\alpha]-d_a[w,\phi]$.

Next we show that $[\tilde f,\tilde \alpha]$ satisfies \eqref{eqn::modifying f to have zero jet}. By the assumption $\bI_a [f,\alpha]=0$ we have $\tbN_a [\tilde f,\tilde \alpha]=0$ in $\widetilde M^{\rm int}$. Hence, $\tbN_a [\tilde f,\tilde \alpha]$ is smooth near $M$. One can check that $\tbN_a$ is elliptic for $\xi_n\neq 0$ if it is restricted to the pairs satisfying \eqref{eqn::tilde f_in=tilde alpha_n=0}. This, in particular, gives that $N^*(\p M)\cap \WF(\cE_{\widetilde M}[\tilde f,\tilde \alpha])=\varnothing$. Since $\cE_{\widetilde M}[\tilde f,\tilde \alpha]=0$ in $\widetilde M\setminus M$, we get that $\p^m_{x^n}[\tilde f,\tilde \alpha]|_{x^n=0}=0$ for all $m\ge 0$. This implies \eqref{eqn::modifying f to have zero jet} as desired.
\end{proof}

Now, we are ready to prove Theorem~\ref{thm::main 1}.

\begin{proof}[Proof of Theorem~\ref{thm::main 1}]
Suppose that $[f,\alpha]\in L^2(M;S^2_M\times \Lambda^1_M)$ is in the kernel of $\bI_a$ and satisfies $[f,\alpha]=\cS_a [f,\alpha]$. Then $[f,\alpha]\in\cA(M;S^2_M\times \Lambda^1_M)$ by Lemma~\ref{lemma::If=0 implies f is analytic}. According to Lemma~\ref{lemma::modifying f to have zero jet}, we can find $[w,\phi]\in C^\infty(M;\Lambda^1_M\times\C)$, with $[w,\phi]|_{\p M}=0$, such that $[\tilde f,\tilde \alpha]:=[f,\alpha]-d_a[w,\phi]$ satisfies \eqref{eqn::modifying f to have zero jet} and \eqref{eqn::tilde f_in=tilde alpha_n=0}. Since $a$ and $[f,\alpha]$ are real analytic on $M$, and $[w,\phi]$ solves the equation \eqref{eqn::system of ODEs} in the boundary normal coordinates, then $[w,\phi]$, and hence, $[\tilde f,\tilde \alpha]$ are real analytic near $\p M$. The pair $[\tilde f,\tilde \alpha]$ vanishes in a neighborhood of $\p M$ in $M$, since $[\tilde f,\tilde \alpha]$ vanishes to infinite order on $\p M$. Therefore, $[f,\alpha]=d_a[w,\phi]$ near $\p M$.

Now, following exactly the same approach as the proof of \cite[Theorem~1]{stefanov2008nonsimple}, we can show that $[w,\phi]$ admits an analytic continuation $[w_0,\phi_0]\in\cA(M;\Lambda^1_M\times\C)$ from a neighborhood of $\p M$ in $M$ to $M$ such that $[f,\alpha]=d_a[w_0,\phi_0]$. Since $[w_0,\phi_0]|_{\p M}=0$ and $[f,\alpha]=\cS_a [f,\alpha]$, this gives that $[f,\alpha]=0$.
\end{proof}

%\section{Reducing triples to pairs}
%Note that $\cI^2_a[f,\alpha,\psi]=\bI_a[f+\psi g,\alpha]$. Therefore, if $[f,\alpha,\psi]\in \ker \cI^2_a$, then $[f+\psi g,\alpha]=d_a[w,\phi]$ for some $[w,\phi]\in H^1(M;\Lambda^1_M\times\C)$ with $[w,\phi]|_{\p M}=0$. In other words, $f+\psi g=d^sw+a\phi g$ and $\alpha=d\phi+aw$. The former one gives
%$$
%f_{ij}(x) v^i v^j+\psi(x)=(d^s w)_{ij}(x) v^i v^j+a(x)\phi(x)\quad\text{for}\quad (x,v)\in SM.
%$$
%Extending this to $TM\setminus\{0\}$, we obtain
%$$
%f_{ij}(x) \xi^i \xi^j+\psi(x)=(d^s w)_{ij}(x) \xi^i \xi^j+a(x)\phi(x)\quad\text{for}\quad (x,\xi)\in TM\setminus\{0\}.
%$$
%Differentiating with respect to $\xi$ twice, we get $f=d^s w$. From this, we also obtain $\psi=a\phi$.

\bibliographystyle{abbrv}
\bibliography{../../Bibliography/Bibliography}

\begin{thebibliography}{10}

\bibitem{abhishek2017support}
A.~Abhishek and R.~K. Mishra.
\newblock Support theorems and an injectivity result for integral moments of a
  symmetric $m$-tensor field.
\newblock {\em preprint arXiv:1704.02010}, 2017.

\bibitem{ainsworth2013attenuated}
G.~Ainsworth.
\newblock The attenuated magnetic ray transform on surfaces.
\newblock {\em Inverse Problems \& Imaging}, {\bf 7}(1):27--46, 2013.

\bibitem{ainsworth2015range}
G.~Ainsworth and Y.~M. Assylbekov.
\newblock On the range of the attenuated magnetic ray transform for connections
  and {H}iggs fields.
\newblock {\em Inverse Problems \& Imaging}, {\bf 9}(2), 2015.

\bibitem{assylbekov2017polyharmonicadmissible}
Y.~M. Assylbekov and Y.~Yang.
\newblock Determining the first order perturbation of a polyharmonic operator
  on admissible manifolds.
\newblock {\em Journal of Differential Equations}, {\bf 262}(1):590--614, 2017.

\bibitem{assylbekov2017kerrinverse}
Y.~M. Assylbekov and T.~Zhou.
\newblock Direct and inverse problems for the nonlinear time-harmonic {M}axwell
  equations in {K}err-type media.
\newblock {\em preprint, arXiv:1709.07767}, 2017.

\bibitem{budinger1979emission}
T.~Budinger, G.~Gullberg, and R.~Huesman.
\newblock Emission computed tomography.
\newblock {\em Image reconstruction from projections}, pages 147--246, 1979.

\bibitem{chung2017hodge}
F.~J. Chung, M.~Salo, and L.~Tzou.
\newblock Partial data inverse problems for the {H}odge {L}aplacian.
\newblock {\em Analysis \& PDE}, {\bf 10}(1):43--93, 2017.

\bibitem{dairbekov2007boundary}
N.~S. Dairbekov, G.~P. Paternain, P.~Stefanov, and G.~Uhlmann.
\newblock The boundary rigidity problem in the presence of a magnetic field.
\newblock {\em Advances in Mathematics}, {\bf 216}(2):535--609, 2007.

\bibitem{dairbekov2011conformal}
N.~S. Dairbekov and V.~A. Sharafutdinov.
\newblock On conformal {K}illing symmetric tensor fields on {R}iemannian
  manifolds.
\newblock {\em Siberian Advances in Mathematics}, {\bf 21}(1):1--41, 2011.

\bibitem{duvaut1976inequalities}
G.~Duvaut and J.~L. Lions.
\newblock {\em Inequalities in mechanics and physics}.
\newblock Springer, 1976.

\bibitem{ferreira2009limiting}
D.~D.~S. Ferreira, C.~E. Kenig, M.~Salo, and G.~Uhlmann.
\newblock Limiting {C}arleman weights and anisotropic inverse problems.
\newblock {\em Inventiones Mathematicae}, {\bf 178}(1):119--171, 2009.

\bibitem{frigyik2008x}
B.~Frigyik, P.~Stefanov, and G.~Uhlmann.
\newblock The {X}-ray transform for a generic family of curves and weights.
\newblock {\em Journal of Geometric Analysis}, {\bf 18}(1):89--108, 2008.

\bibitem{ghosh2017inverse}
T.~Ghosh and S.~Bhattacharyya.
\newblock Inverse boundary value problem of determining up to second order
  tensors appear in the lower order perturbations of the polyharmonic operator.
\newblock {\em preprint arXiv:1706.03823}, 2017.

\bibitem{guillarmou2016negconnections}
C.~Guillarmou, G.~P. Paternain, M.~Salo, and G.~Uhlmann.
\newblock The {X}-ray transform for connections in negative curvature.
\newblock {\em Communications in Mathematical Physics}, 343(1):83--127, 2016.

\bibitem{guillemin1979some}
V.~Guillemin and S.~Sternberg.
\newblock Some problems in integral geometry and some related problems in
  micro-local analysis.
\newblock {\em American Journal of Mathematics}, {\bf 101}(4):915--955, 1979.

\bibitem{holman2013generic}
S.~Holman.
\newblock Generic local uniqueness and stability in polarization tomography.
\newblock {\em The Journal of Geometric Analysis}, {\bf 1}(23):229--269, 2013.

\bibitem{holman2010doppler}
S.~Holman and P.~Stefanov.
\newblock The weighted {D}oppler transform.
\newblock {\em Inverse Problems \& Imaging}, {\bf 4}(1):111--130, 2010.

\bibitem{juhlin1992principles}
P.~Juhlin.
\newblock Principles of {D}oppler tomography.
\newblock {\em LUTFD2/(TFMA-92)/7002 P}, {\bf 17}, 1992.

\bibitem{kenig2011inverse}
C.~E. Kenig, M.~Salo, and G.~Uhlmann.
\newblock Inverse problems for the anisotropic {M}axwell equations.
\newblock {\em Duke Mathematical Journal}, {\bf 157}(2):369--419, 2011.

\bibitem{krupchyk1702inverse}
K.~Krupchyk and G.~Uhlmann.
\newblock Inverse problems for magnetic schr{\"o}dinger operators in
  transversally anisotropic geometries, preprint (2017).
\newblock {\em preprint arXiv:1702.07974}.

\bibitem{krupchyk2017inverse}
K.~Krupchyk and G.~Uhlmann.
\newblock Inverse problems for advection diffusion equations in admissible
  geometries.
\newblock {\em preprint arXiv:1704.05598}, 2017.

\bibitem{mclean2000strongly}
W.~C.~H. McLean.
\newblock {\em Strongly elliptic systems and boundary integral equations}.
\newblock Cambridge University Press, 2000.

\bibitem{melrose1994spectral}
R.~B. Melrose.
\newblock Spectral and scattering theory for the laplacian on asymptotically
  euclidian spaces.
\newblock {\em Lecture Notes in Pure and Applied Mathematics}, pages 85--85,
  1994.

\bibitem{michel1981rigidite}
R.~Michel.
\newblock Sur la rigidit{\'e} impos{\'e}e par la longueur des
  g{\'e}od{\'e}siques.
\newblock {\em Inventiones Mathematicae}, {\bf 65}(1):71--83, 1981.

\bibitem{monard2016inversion}
F.~Monard.
\newblock Inversion of the attenuated geodesic {X}-ray transform over functions
  and vector fields on simple surfaces.
\newblock {\em SIAM Journal on Mathematical Analysis}, {\bf 48}(2):1155--1177,
  2016.

\bibitem{monard2017efficient}
F.~Monard.
\newblock Efficient tensor tomography in fan-beam coordinates. {II}:
  {A}ttenuated transforms.
\newblock {\em to appear in Inverse Problems \& Imaging, arXiv:1704.08294},
  2017.

\bibitem{morrey1957analyticity}
C.~Morrey and L.~Nirenberg.
\newblock On the analyticity of the solutions of linear elliptic systems of
  partial differential equations.
\newblock {\em Communications on Pure and Applied Mathematics}, {\bf
  10}(2):271--290, 1957.

\bibitem{paternain2012attenuated}
G.~P. Paternain, M.~Salo, and G.~Uhlmann.
\newblock The attenuated ray transform for connections and {H}iggs fields.
\newblock {\em Geometric and Functional Analysis}, {\bf 22}(5):1460--1489,
  2012.

\bibitem{paternain2013tensor}
G.~P. Paternain, M.~Salo, and G.~Uhlmann.
\newblock Tensor tomography on surfaces.
\newblock {\em Inventiones Mathematicae}, {\bf 193}(1):229--247, 2013.

\bibitem{paternain2014tensor}
G.~P. Paternain, M.~Salo, and G.~Uhlmann.
\newblock Tensor tomography: Progress and challenges.
\newblock {\em Chinese Annals of Mathematics, Series B}, {\bf 35}(3), 2014.

\bibitem{paternain2015invariant}
G.~P. Paternain, M.~Salo, and G.~Uhlmann.
\newblock Invariant distributions, {B}eurling transforms and tensor tomography
  in higher dimensions.
\newblock {\em Mathematische Annalen}, {\bf 363}(1-2):305--362, 2015.

\bibitem{paternain2016geodesic}
G.~P. Paternain, M.~Salo, G.~Uhlmann, and H.~Zhou.
\newblock The geodesic {X}-ray transform with matrix weights-ray transform with
  matrix weights.
\newblock {\em preprint arXiv:1605.07894}, 2016.

\bibitem{sadiq2016tensortransform}
K.~Sadiq, O.~Scherzer, and A.~Tamasan.
\newblock On the {X}-ray transform of planar symmetric 2-tensors.
\newblock {\em Journal of Mathematical Analysis and Applications}, {\bf
  442}(1):31--49, 2016.

\bibitem{salo2011attenuated}
M.~Salo and G.~Uhlmann.
\newblock The attenuated ray transform on simple surfaces.
\newblock {\em Journal of Differential Geometry}, {\bf 88}(1):161--187, 2011.

\bibitem{sharafutdinov2007variations}
V.~Sharafutdinov.
\newblock Variations of {D}irichlet-to-{N}eumann map and deformation boundary
  rigidity of simple 2-manifolds.
\newblock {\em Journal of Geometric Analysis}, {\bf 17}(1):147--187, 2007.

\bibitem{sharafutdinov1994integral}
V.~A. Sharafutdinov.
\newblock {\em Integral geometry of tensor fields}.
\newblock VSP, Utrecht, The Netherlands, 1994.

\bibitem{shubin2001pseudodifferential}
M.~Shubin.
\newblock {\em Pseudodifferential {O}perators and {S}pectral {T}heory}.
\newblock Springer Science \& Business Media, 2001.

\bibitem{sjostrand1982singularites}
J.~Sj{\"o}strand.
\newblock {\em Singularit{\'e}s analytiques microlocales}, volume~82.
\newblock Soci{\'e}t{\'e} Math{\'e}matique de France, 1982.

\bibitem{stefanov2008microlocal}
P.~Stefanov.
\newblock Microlocal approach to tensor tomography and boundary and lens
  rigidity.
\newblock {\em Serdica Mathematical Journal}, {\bf 34}(1):67p--112p, 2008.

\bibitem{stefanov2008sharp}
P.~Stefanov.
\newblock A sharp stability estimate in tensor tomography.
\newblock In {\em Journal of Physics: Conference Series}, volume~{\bf 124},
  page 012007. IOP Publishing, 2008.

\bibitem{stefanov2004stability}
P.~Stefanov and G.~Uhlmann.
\newblock Stability estimates for the {X}-ray transform of tensor fields and
  boundary rigidity.
\newblock {\em Duke Mathematical Journal}, {\bf 123}(3):445--467, 2004.

\bibitem{stefanov2005boundary}
P.~Stefanov and G.~Uhlmann.
\newblock Boundary rigidity and stability for generic simple metrics.
\newblock {\em Journal of the American Mathematical Society}, {\bf
  18}(4):975--1003, 2005.

\bibitem{stefanov2008boundary}
P.~Stefanov and G.~Uhlmann.
\newblock Boundary and lens rigidity, tensor tomography and analytic microlocal
  analysis.
\newblock {\em Algebraic Analysis of Differential Equations, Fetschrift in
  Honor of Takahiro Kawai, edited by T. Aoki, H. Majima, Y. Katei and N. Tose},
  pages 275--293, 2008.

\bibitem{stefanov2008nonsimple}
P.~Stefanov and G.~Uhlmann.
\newblock Integral geometry of tensor fields on a class of non-simple
  {R}iemannian manifolds.
\newblock {\em American Journal of Mathematics}, {\bf 130}(1):239--268, 2008.

\bibitem{stefanov2014inverting}
P.~Stefanov, G.~Uhlmann, and A.~Vasy.
\newblock Inverting the local geodesic {X}-ray transform on tensors.
\newblock {\em preprint, arXiv:1410.5145}, 2014.

\bibitem{taylor2011partial}
M.~Taylor.
\newblock {\em Partial differential equations {I: B}asic theory}, volume~{\bf
  115} of {\em Applied Mathematical Sciences}.
\newblock Springer-Verlag New York, 2011.

\bibitem{taylor1981pseudodifferential}
M.~E. Taylor.
\newblock Pseudodifferential operators.
\newblock volume 34 of Princeton Mathematical Series, 1981.

\bibitem{treves1980introduction}
F.~Tr{\`e}ves.
\newblock {\em Introduction to pseudodifferential and {F}ourier integral
  operators}.
\newblock Springer Science \& Business Media, 1980.

\bibitem{uhlmann2016inverse}
G.~Uhlmann and A.~Vasy.
\newblock The inverse problem for the local geodesic ray transform.
\newblock {\em Inventiones Mathematicae}, {\bf 205}(1):83--120, 2016.

\bibitem{zhou2017generic}
H.~Zhou.
\newblock Generic injectivity and stability of inverse problems for
  connections.
\newblock {\em Communications in Partial Differential Equations}, {\bf
  42}(5):780--801, 2017.

\end{thebibliography}
\end{document}